\documentclass[11pt,oneside]{amsart}
	\usepackage[a4paper]{geometry}
	\usepackage{amsmath}
	\usepackage{amsthm}
	\usepackage{amssymb}
	\usepackage{mathtools}
	\usepackage{amsfonts}
	\usepackage{relsize}
	\usepackage{float}
	\usepackage[myheadings]{fullpage}
	\usepackage{fancyhdr}
	\usepackage{enumerate}
	\usepackage{verbatim}
	\usepackage{hyperref}
	\usepackage{lastpage}
	\usepackage{graphicx, wrapfig, setspace, booktabs}
	\usepackage[T1]{fontenc}
	\usepackage{lmodern}
	\usepackage{url, lipsum}
        \usepackage{xcolor}
	\DeclareFontFamily{U}{matha}{\hyphenchar\font45}
\DeclareFontShape{U}{matha}{m}{n}{
      <5> <6> <7> <8> <9> <10> gen * matha
      <10.95> matha10 <12> <14.4> <17.28> <20.74> <24.88> matha12
      }{}
\DeclareSymbolFont{matha}{U}{matha}{m}{n}
\DeclareFontSubstitution{U}{matha}{m}{n}

\usepackage{tikz}
\usetikzlibrary{trees}

\DeclareFontFamily{U}{mathx}{\hyphenchar\font45}
\DeclareFontShape{U}{mathx}{m}{n}{
      <5> <6> <7> <8> <9> <10>
      <10.95> <12> <14.4> <17.28> <20.74> <24.88>
      mathx10
      }{}
\DeclareSymbolFont{mathx}{U}{mathx}{m}{n}
\DeclareFontSubstitution{U}{mathx}{m}{n}

\DeclareMathDelimiter{\vvvert}{0}{matha}{"7E}{mathx}{"17}

	\newtheorem{thm}{Theorem}[section]
	\newtheorem{lem}[thm]{Lemma}
	\newtheorem{prop}[thm]{Proposition}
	\newtheorem{cor}[thm]{Corollary}
	\theoremstyle{definition}
	\newtheorem{df}{Definition}
	\newtheorem{rem}{Remark}
	\newtheorem{eg}{Example}
	\newtheorem{problem}{Problem}

	{\setlength{\parindent}{0cm}
	
	\onehalfspacing
	\setcounter{tocdepth}{5}
	\setcounter{secnumdepth}{5}
	\pagestyle{fancy}
	\setlength\headheight{15pt}
	\fancyhead[L]{\leftmark}

\title[On the spaces dual to combinatorial Banach spaces]{On the spaces dual to combinatorial Banach spaces}

\author[P. Borodulin-Nadzieja]{Piotr Borodulin--Nadzieja}
\address[Piotr Borodulin-Nadzieja]{Mathematical Institute, University of Wroc\l aw \\  pl. Grunwaldzki 2, 50-384 Wroc\l aw, Poland}
\email{pborod@math.uni.wroc.pl}

\author[S. Jachimek]{Sebastian Jachimek}
\address[Sebastian Jachimek]{Mathematical Institute, University of Wroc\l aw \\  pl. Grunwaldzki 2, 50-384 Wroc\l aw, Poland}
\email{sebastian.jachimek@math.uni.wroc.pl}

\author[A. Pelczar-Barwacz]{Anna Pelczar-Barwacz}
\address[Anna Pelczar-Barwacz]{Institute of Mathematics, Jagiellonian University \\ ul. prof. Stanisława Łojasiewicza 6, 30-348 Krak\o w, Poland}
\email{anna.pelczar@uj.edu.pl}

\thanks{The first author was supported by the 
National Science Center project no. 2018/29/B/ST1/00223. The second author was supported by the National Science Centre, Poland under
the Weave-UNISONO call in the Weave programme 2021/03/Y/ST1/00124. The third author was supported by the grant of the National Science Centre, Poland, no. UMO-2020/39/B/ST1/01042.}

\begin{document}

\maketitle
\begin{abstract}
	We present quasi-Banach spaces which are closely related to the duals of combinatorial Banach spaces. More precisely, for a compact family $\mathcal{F}$
	of finite subsets of $\omega$ we define a quasi-norm $\lVert \cdot \rVert^\mathcal{F}$ whose Banach envelope 
 is the dual norm for the combinatorial space generated by $\mathcal{F}$. Such quasi-norms seem to be much easier to handle than the dual norms and yet the quasi-Banach spaces induced by them share many properties with the dual spaces. We show that the quasi-Banach spaces  induced by large families (in the sense of Lopez-Abad and Todorcevic) are
	$\ell_1$-saturated and do not have the Schur property. In particular, this holds for the Schreier families. 
\end{abstract} 
\section{Introduction}

A combinatorial Banach space is the completion of $c_{00}$, the space of all finitely supported sequences of real numbers, with respect to the following norm

\begin{equation}\label{dd1}
\lVert x \rVert_\mathcal{F} = \sup_{F\in \mathcal{F}} \sum_{k\in F} |x(k)|,
\end{equation}

where $\mathcal{F}$ is a family of finite subsets of $\omega$ (which is hereditary and covers $\omega$, see Section \ref{schreier-like} for the details). We denote such completion by $X_\mathcal{F}$. One of the most 'famous' examples of a Banach
space of this form is the \emph{Schreier space} induced by so-called Schreier family $\mathcal{S}$ (see Section \ref{schreier-like}) and so combinatorial Banach spaces are sometimes called Schreier-like spaces. They were studied mainly for $\mathcal{F}$ being compact families, e.g. by  Castillo and Gonzales (\cite{Castillo-Gonzales}), Lopez Abad and Todorcevic (\cite{Jordi-Stevo}), Brech, Ferenczi and Tcaciuc (\cite{Brech}). Recall that a family
$\mathcal{F} \subseteq \mathcal{P}(\omega)$ is compact if it is compact as a subset of $2^\omega$, via the natural identification of $\mathcal{P}(\omega)$ and $2^\omega$.
In \cite{NaFa} Borodulin-Nadzieja and Farkas considered Schreier-like spaces for families that are not necessarily compact and that were motivated by the theory of analytic P-ideals.

The main motivation of our article is an attempt to study Banach spaces dual to combinatorial Banach spaces. Even in the case of Schreier space, not much seems to be known about its dual (see \cite{Lipieta}). Perhaps the reason 
lies in the lack of a nice description of the dual norm. Seeking such a description we came up with the following function $\lVert \cdot\rVert^{\mathcal{F}} \colon c_{00} \to \mathbb{R}$:

\begin{equation} \label{dd2}
	\lVert x \rVert^\mathcal{F} = \inf\{ \sum_{F \in \mathcal{P}} \sup_{i\in F} |x(i)|\colon \mathcal{P}\subseteq \mathcal{F} \mbox{ is a partition of }\omega\}.
\end{equation}
Perhaps this formula does not look tempting at first glance, but in a 'combinatorial' sense it is dual to $\lVert \cdot \rVert_\mathcal{F}$. Indeed, we can think about evaluating $\lVert x \rVert_{\mathcal{F}}$ as partitioning $\omega$ into pieces from
$\mathcal{F}$, summing up $|x(i)|$ for $i$ from one piece of the partition and then maximizing the result, for all partitions and all pieces. On the other hand, evaluating $\lVert x \rVert^\mathcal{F}$ comes down to partitioning $\omega$ into pieces from $\mathcal{F}$, taking a
maximum of $|x(i)|$ for
$i$ from one piece of the partition, summing up those maxima and then minimizing the result for all possible partitions.

For certain families $\mathcal{F}$ the completion of $c_{00}$ in $\lVert \cdot \rVert^\mathcal{F}$ (which we will denote by $X^\mathcal{F}$) is a Banach space which is isometrically isomorphic to $X^*_\mathcal{F}$ (see Proposition \ref{for-partitions}). E.g. notice that if $\mathcal{F}$ consists of singletons, then $X_\mathcal{F}$ is isometrically isomorphic to $c_0$ and $X^\mathcal{F}$ is isometrically isomorphic to $\ell_1$. 

In general, if a family $\mathcal{F} \subseteq [\omega]^{<\omega}$ is compact, hereditary and covering $\omega$, then $X^\mathcal{F}$ shares many properties with $X^*_\mathcal{F}$. In fact, it was difficult for us to find a property distinguishing those
spaces apart from the fact that for many families $\mathcal{F}$ the space
$X^\mathcal{F}$ is not a Banach space! More precisely, $\lVert \cdot \rVert^\mathcal{F}$ does not need to satisfy the triangle inequality (and it does not e.g. for $\mathcal{F}$ being the Schreier family).

The lack of triangle inequality is not something welcomed in the theory of Banach spaces. However, there are still good reasons to study $\lVert \cdot \rVert^\mathcal{F}$.
First, if $\mathcal{F}\subseteq [\omega]^{<\omega}$ is compact, hereditary and covering $\omega$, then $\lVert \cdot
\rVert^\mathcal{F}$ enjoys some perverted form of triangle inequality. Namely, for every $x,y \in c_{00}$
\[ \lVert x+y \rVert^\mathcal{F} \leq 2 (\lVert x \rVert^\mathcal{F} + \lVert y \rVert^\mathcal{F}) \]
and so $\lVert \cdot \rVert^\mathcal{F}$ is a quasi-norm (see Section \ref{quasi-banach} for the
definition and basic facts concerning quasi-norms). We show that $X^\mathcal{F}$ is a quasi-Banach space. 

Second, the space $X^\mathcal{F}$ is indeed closely related to $X^*_\mathcal{F}$, as we show that $X^*_\mathcal{F}$ is the Banach envelope of $X^\mathcal{F}$ for compact $\mathcal{F}$ (Theorem \ref{final_isomorphism}). This notion was introduced independently by \cite{Peetre} and \cite{Shapiro} and studied extensively. The Banach envelope of a quasi-normed space $(X,\lVert\cdot\rVert)$ is the completion of $X$ with the norm given by the formula \[ \vvvert x \vvvert = \inf\Big\{\sum_{i\leq n} \lVert x_i \rVert\colon x = \sum_{i\leq n} x_i , n\in\omega\Big\}, \ \ x\in X. \] 
The norm defined above is clearly the largest norm on $X$ dominated by the quasi-norm $\lVert\cdot\rVert$, in other words, it is the Minkowski functional of the convex hull of the unit ball in $(X,\lVert\cdot\rVert)$. Alternatively, the Banach envelope can be described as the closure in $X^{**}$ of the canonical image of $X$. It follows that the dual spaces of $X$ and its Banach envelope coincide. For these and related facts see \cite[Chapter 2.4]{Kalton-F}. 

We shall prove that the space $X^\mathcal{F}$ can be embedded into $X^*_\mathcal{F}$ by a bounded operator $T_0$ of norm 1, that transports the set of extreme points of the unit ball of $X^\mathcal{F}$ onto the set of extreme points of the unit ball of $X^*_\mathcal{F}$   (see Proposition  \ref{extreme_to_extreme}). This fact implies that $X^*_\mathcal{F}$ is the Banach envelope of $X^\mathcal{F}$.

It is worth taking a look from the geometric point of view. The unit ball in $X^\mathcal{F}$ is typically non-convex. However, this ball and the unit ball in $X^*_\mathcal{F}$ have the same extreme points, and the closed convex hull of the ball in $X^\mathcal{F}$ is \emph{exactly} the unit ball of $X^*_\mathcal{F}$. 

The fact that $X^*_\mathcal{F}$ is the Banach envelope of $X^\mathcal{F}$ does not mean that these spaces are necessarily isomorphic. They are  if and only if there is $C>0$ such that for every sequence $(x_i)_{i\leq n}$
\[ \lVert \sum_{i\leq n} x_i \rVert^\mathcal{F} \leq C \sum_{i\leq n} \lVert x_i \rVert^\mathcal{F}. \]
As we have already mentioned, there are examples for which $X^\mathcal{F}$ is isometrically isomorphic to $X^*_\mathcal{F}$ (Proposition \ref{for-partitions}). On the other hand, it seems that typically those spaces are not isomorphic. In Example \ref{tree} we show a simple example of a compact family $\mathcal{F}$ and a set of vectors $(x_i)$ in $X^\mathcal{F}$ for which there is no $C>0$ as above. Then we show that 
the space $X^\mathcal{S}$ is not isomorphic to $X^*_\mathcal{S}$, for the Schreier family $\mathcal{S}$.

However, even if the spaces $X^\mathcal{F}$ and $X^*_\mathcal{F}$ are not isomorphic, the connection between $X^\mathcal{F}$ and $X^*_\mathcal{F}$ is so strong that they share properties that are typically not invariant under isomorphism. As we have mentioned the unit balls in those spaces have 'the same' extreme points. Also, for every compact hereditary family $\mathcal{F}$ the space $X^\mathcal{F}$ and $X^*_\mathcal{F}$ have isometrically isomorphic duals (Theorem \ref{isomorphism}).

We will argue that the quasi-norm $\lVert \cdot \rVert ^\mathcal{F}$ is much easier to handle than the dual norm of $X^*_\mathcal{F}$.

It seems that the space $X^*_\mathcal{S}$, the dual to the Scheier space, is a rather mysterious object and not much of it is known. Recently, Galego, Gonz\'{a}lez and Pello
(\cite{Galego}) proved that it is $\ell_1$-saturated (i.e. every closed infinite-dimensional subspace contains an isomorphic copy of
$\ell_1$) and it is known that it does not enjoy the Schur
property. In fact, for years it was an open problem whether such spaces, $\ell_1$-saturated but without Schur property, exist at all. 
The first example was constructed by Bourgain (\cite{Bourgain}) and then several other involved constructions of such spaces have been presented (see e.g. \cite{Hagler}, \cite{Popov}). The result of Galego, Gonz\'{a}lez and Pello says that you do not have to
\emph{construct} an
$\ell_1$-saturated space without a Schur property: the dual to Schreier space is already a good example (although the proof contained in \cite{Galego} is rather difficult)\footnote{The authors of \cite{Galego} prove this result en passant. They
do not mention it neither in the abstract nor in the introduction, nor they relate it
to the Bourgain construction. We are grateful to
Jes\'us Castillo for informing us about this result.}. Moreover, recent results of Causey and the third author (\cite{CauseyP}) imply the $\ell_1$-saturation of dual spaces of Schreier spaces of arbitrary order.

We show that the space $X^\mathcal{S}$ is also $\ell_1$-saturated (Theorem \ref{ell1}) and it does not have Schur property (Proposition \ref{Schur}). The proof is considerably simpler than those from \cite{Galego} and the previous constructions. In
fact, our proof works for all combinatorial Banach spaces generated by large families (see Section \ref{l1-schur} for the precise definition of largeness), in particular for the Schreier spaces of all orders. The assumption of largeness is rather
mild and so, although for some time it was not
clear if $\ell_1$-saturated spaces without Schur property exist at all, we show that the duals to combinatorial spaces \emph{typically} satisfies the conjunction of those properties. 

Of course, as we do not know if $X^\mathcal{F}$ is isomorphic to $X^*_\mathcal{F}$ (and as usually they are not), we cannot claim that the above result says something about $X^*_\mathcal{F}$ directly. Perhaps the relation between quasi-Banach spaces and their Banach envelopes is worth studying in general. E.g. one can ask the following question. 
\begin{problem}
Let $(X, \lVert \cdot \rVert)$ be a quasi-Banach space, and let $Y$ be its Banach envelope. %, i.e. let $Y$ be the Banach space normed by \[ \vvvert x \vvvert  = \inf\Big\{\sum_{i\leq n} \lVert x_i \rVert\colon x = \sum_{i\leq n} x_i \Big\}. \]
Which properties of $X$ are shared by $Y$? In particular, if $X$ is $\ell_1$-saturated, is $Y$ $\ell_1$-saturated as well? 
\end{problem}
Note here that a Banach envelope of a non-locally convex quasi-Banach space contains $\ell_1^n$'s uniformly by \cite{Kalton}.

The article is organized as follows. In Section \ref{preliminaries} we introduce the basic notions and recall some facts concerning quasi-Banach spaces. In Section \ref{schreier-like}  we present basic definitions and facts about combinatorial Banach
spaces. In Section \ref{dual-norms} we deal with
the function defined by (\ref{dd2}).  We prove that $\lVert \cdot \rVert^{\mathcal{F}}$ is a nice quasi-norm and that $X^{\mathcal{F}}$ is a quasi-Banach space for every compact family $\mathcal{F}$ (which is hereditary and which covers $\omega$).
In
Section \ref{how-close} we prove that $X^\mathcal{F}$ is isometrically isomorphic to $X^*_\mathcal{F}$ for every compact hereditary family $\mathcal{F}$ of finite subsets of $\omega$. In Section \ref{l1-schur} we prove that $X^{\mathcal{F}}$ is
$\ell_{1}$-saturated and that it does not have the Schur property, for $\mathcal{F}$ being large, in particular for the Schreier families of any order.
%Finally, in Section \ref{extremep} we point out more connections between $X^\mathcal{F}$ and $X^*_\mathcal{F}$; we show that they have the same extreme points.

\section{Acknowledgements}

	We thank Barnab\'as Farkas, and Alberto Salguero Alarc\'on for the discussions concerning the subject of this article. We are grateful to the anonymous referee for valuable suggestions.

\section{Preliminaries}\label{preliminaries}

	In this section we present standard facts and notations, used in the subsequent sections.

	By $\omega$ we denote the set of natural numbers. For various technical reasons for the sake of this article we will not recognize 0 as a natural number. If $k \in \omega$ and $X \subseteq \omega$, then $[X]^{\leq k} = \lbrace A \subseteq X: \lvert A \rvert \leq k \rbrace$. Similarly, $[X]^{k}$ denotes the family of all subsets of $X$ with $k$ elements and $[X]^{< \omega}$ is the family of all finite subsets of $X$.

%	Recall that the Cantor set $2^{\omega}$ is a compact and complete metric space equipped with the metric 
%	$$
%	d(x,y) = \dfrac{1}{2^{k}},
%	$$
%	where $k$ is the smallest natural number for which $x(k) \neq y(k)$ (and $d(x,y)=0$ in case $x=y$).
	We will consider the standard topology on the Cantor set. 
	By $\mathcal{P}(\omega)$ we denote the power set of $\omega$ which we identify with the Cantor set $2^{\omega}$. 
	%via the map 
	%$$\mathcal{P}(\omega) \ni A \mapsto \chi_{A} \in 2^{\omega},$$
	%where $\chi_{A}$ denotes the characteristic function of a set $A$.
	Henceforth we do not distinguish a subset of $\omega$ from its characteristic function, so if we write that a family is open, closed, compact, etc. we mean openness, closedness, and compactness of a subset of $2^{\omega}$. In addition, we use interchangeably the following notations
	\[
	j \in A \hspace{0.1cm} (\notin A)
	\]
	and
	\[
	A(j) = 1 \hspace{0.1cm} (=0).
	\]

	If $A,B \subseteq \omega$ are two finite sets, then by $A < B$ we mean that $\max(A) < \min(B)$.  In what follows $\mathcal{F}$ will always denote a family of finite subsets of $\omega$ which is hereditary (i.e. it is closed under taking subsets)
	and which covers
	$\omega$. If $\mathcal{A}$ is a family of subsets of $\omega$ then by its \emph{hereditary closure}  we mean the smallest hereditary family containing $\mathcal{A}$.
	
\begin{df} A family $\mathcal{P}\subseteq \mathcal{P}(\omega)$ is a \emph{partition} if 
		\begin{itemize}
			\item $\emptyset\in \mathcal{P}$,
			\item $\bigcup \mathcal{P} = \omega$,
			\item $\mathcal{P}$ consists of pairwise disjoint elements.
		\end{itemize}
	\end{df}
	Notice that our definition of partition is non-standard. We want $\emptyset$ to be an element of a partition, since then it forms a compact subset of $2^\omega$ and we are going to use this fact later on. 
	We denote by $\mathbb{P}_{\mathcal{F}}$ the family of all partitions $\mathcal{P}\subseteq \mathcal{F}$.

For $x \in \mathbb{R}^{\omega}$ by its \textit{support} we mean the set $\mathrm{supp}(x) = \lbrace n \in \omega: x(n) \neq 0 \rbrace$ and $c_{00}$ is the set of all sequences $x$ for which $\mathrm{supp}(x)$ is finite.   
	For $A \subseteq \omega$ the function $P_{A}: \mathbb{R}^{\omega} \rightarrow \mathbb{R}^{\omega}$ given by 
	\[
	P_{A}(x)(k) = 
	\begin{cases}
	x(k), & \text{if } k \in A, \\
	0, & \text{otherwise,}
	\end{cases}
	\]
	is called the \textit{projection} of $x$ onto coordinates from $A$. In particular if $n \in \omega$, then by $P_{n}$ we denote the projection onto the initial segment $\lbrace 1,2,...,n \rbrace$ and by $P_{\omega \setminus n}$ - the projection onto
	the complement of this segment.

\subsection{Quasi-Banach spaces} \label{quasi-banach}
\begin{df} \label{quasi_norm_def}
A \emph{quasi-norm} on a vector space $X$ is a function $\lVert \cdot \rVert: X \rightarrow \mathbb{R}$ satysfing 
\begin{itemize}
\item $\lVert x \rVert = 0 \Leftrightarrow x = 0$,
\item For every  $\lambda \in \mathbb{R}$ $\lVert \lambda x \rVert = | \lambda | \lVert x \rVert$,
\item There is $c \geq 1$ such that $\lVert x + y \rVert \leq c ( \lVert x \rVert + \lVert y \rVert )$.
\end{itemize}
\end{df}

The minimal constant $c$ working above is sometimes called the \emph{modulus of concavity} of the quasi-norm. In particular, for $c=1$ we get the definition of a norm.
In what follows sometimes we will allow quasi-norms to take possibly infinite values. If $\lVert \cdot \rVert$ is a quasi-norm (taking only finite values) on a vector
space $X$, then the pair $(X, \lVert \cdot \rVert)$ is called \emph{quasi-normed space}. 

%An important class of quasi-norms are so-called $p$-norms. A quasi-norm is a $p$-norm if it satisfies the following condition:
%\begin{equation} \label{p_norm}
%\lVert x + y \rVert^{p} \leq \lVert x \rVert^{p} + \lVert y \rVert^{p}
%\end{equation}

%The classical Aoki-Rolewicz theorem (\cite{Aoki}, \cite{Rolewicz}, see also \cite[Theorem 1.3]{Kalton-F}), says that every quasi-norm can be renormed to a $p$-norm:

%\begin{thm}[Aoki-Rolewicz] \label{aoki_rolewicz}
%	Let $(X, \lVert \cdot \rVert)$ be a quasi-normed space. Then there exists $p \in (0, 1] $ and $p$-norm $\left \vvvert \cdot \right \vvvert$ on $X$ which is equivalent to $\lVert \cdot \rVert$. If $c$ is a modulus of concavity of $\lVert \cdot
	%\rVert$, then $c = 2^{\frac{1}{p}-1}$.
%\end{thm}

%The appropriate $p$-norm $\left \vvvert \cdot \right \vvvert$ is given by the formula
%\begin{equation} \label{equi_q_norm}
%\left \vvvert x \right \vvvert = \inf \bigg \{ \Big ( \sum \limits_{i=1}^{n} \lVert x_i \rVert^{p} \Big )^{\frac{1}{p}}: n \in \omega, x_1,...,x_n \in X, x = \sum \limits_{i=1}^{n} x_i \bigg \}.
%\end{equation}

%The important consequence of Aoki-Rolewicz theorem is that every quasi-normed space is metrizable. The compatible metric $d$ can be defined by $d(x,y) = \left \vvvert x - y \right \vvvert^{p}$ (see \cite{Kalton_quasi}). A quasi-normed space is called
%a \emph{quasi-Banach} space, if the quasi-norm induces a complete metrizable topology on $X$. 

Note that a quasi-Banach space $X$ which is not isomorphic to a Banach space cannot be locally convex. By this, it turns out that results which hold in Banach spaces and are based somehow on the local convexity (e.g.
Hahn-Banach extension property or Krein-Milman theorem) are no longer true, in general, in quasi-Banach spaces (see \cite{Kalton_quasi}). However, the standard results of Banach space theory such as the Open Mapping Theorem, Uniform Boundedness
Principle, and Closed Graph Theorem can be applied in quasi-Banach spaces since they depend only on the completeness of the space. 
%We will use the Open Mapping Theorem in the context of quasi-Banach spaces, so we recall it below. 
%
%\begin{thm}[Open Mapping Theorem] \label{open_mapping}
%Let $X$ and $Y$ be quasi-Banach spaces and let $T: X \rightarrow Y$ be continuous linear operator onto $Y$. Then $T$ is open. 
%\end{thm}

%{\color{red} czy powtarzac tutaj def Banach envelope ze wstepu?}

\section{Combinatorial spaces} \label{schreier-like}

In \cite{NaFa} the authors presented an approach towards Banach spaces with unconditional bases motivated by the theory of analytic P-ideals on $\omega$. 
We will briefly overview the basic definitions and results which we are going to use in more general case, namely for quasi-Banach spaces. 

\begin{df}\label{nicenorm} \cite{NaFa}
	We say that a function $\varphi\colon \mathbb{R}^{\omega} \rightarrow [0,\infty]$ is \textit{nice}, if 
	\begin{enumerate}[(i)]
		\item (\textit{Non-degeneration})  $\varphi(x) < \infty$ for every $x \in c_{00}$,
	\item (\textit{Monotonicity}) For every $x,y \in \mathbb{R}^\omega$, if $\lvert x(n) \rvert \leq \lvert y(n) \rvert$ for ach each $n \in \omega$, then $\varphi(x) \leq \varphi(y)$,
	\item (\textit{Lower semicontinuity}) $\lim \limits_{n \rightarrow \infty} \varphi(P_{n}(x)) = \varphi(x)$ for every $x\in \mathbb{R}^\omega$.
	\end{enumerate} 
\end{df}
For an extended quasi-norm $\varphi: \mathbb{R}^{\omega} \rightarrow [0,\infty]$ define
	\[
		FIN(\varphi) = \lbrace x \in \mathbb{R}^{\omega}: \varphi(x) < \infty \rbrace,\]
		\[EXH(\varphi) = \lbrace x \in \mathbb{R}^{\omega}: \lim \limits_{n \rightarrow \infty} \varphi(P_{\omega \setminus n}(x)) = 0 \rbrace.\]

\begin{thm}\label{5.1}(\cite[Proposition 5.1]{NaFa}) If $\varphi: \mathbb{R}^{\omega} \rightarrow [0,\infty]$ is a nice extended \emph{norm}, then $FIN(\varphi)$ and $EXH(\varphi)$ (normed with $\varphi$) are Banach spaces, and $EXH(\varphi)$ has an unconditional basis consisting of the unit vectors.
	\end{thm}

	Note that the above theorem can be reversed: every Banach space with an unconditional basis is isometrically isomorphic to $EXH(\varphi)$ for some nice extended norm $\varphi$ (\cite[Proposition 5.1]{NaFa}). Now we will present two structure
	theorems about the spaces $FIN(\varphi)$ and $EXH(\varphi)$. 

	\begin{thm}\label{5.4}(\cite[Theorem 5.4]{NaFa})\label{Banach} Suppose that  $\varphi: \mathbb{R}^{\omega} \rightarrow [0,\infty]$ is a nice extended norm. Then the following conditions are equivalent:
		\begin{itemize}
			\item $EXH(\varphi) = FIN(\varphi)$,
			\item $EXH(\varphi)$ does not contain an isomorphic copy of $c_0$.
		\end{itemize}
	\end{thm}

	\begin{thm}\label{5.5}(\cite[Proposition 5.5]{NaFa}) Suppose that  $\varphi: \mathbb{R}^{\omega} \rightarrow [0,\infty]$ is a nice extended norm. Then the following conditions are equivalent:
		\begin{itemize}
			\item $FIN(\varphi)$ is isometrically isomorphic to $EXH(\varphi)^{**}$,
			\item $EXH(\varphi)$ does not contain an isomorphic copy of $\ell_1$.
		\end{itemize}
	\end{thm}

In this article we consider particular norms induced by families of subsets of $\omega$. More precisely, if $\mathcal{F} \subseteq \mathcal{P}(\omega)$ is hereditary and covering $\omega$, then for a sequence $x$ consider the following expression

	\begin{equation} \label{Lower_norm} 
		\lVert x \rVert_{\mathcal{F}} = \sup_{F \in \mathcal{F}} \sum_{k \in F} \lvert x(k) \rvert.
	\end{equation}
	It is not difficult to check that this is a nice extended norm and so $EXH(\lVert \cdot \rVert_{\mathcal{F}})$ is a Banach space (which we will denote by $X_\mathcal{F}$). See \cite{NaFa} for examples of spaces of this form.
	
	The spaces of the form $X_\mathcal{F}$ are quite well studied for compact families $\mathcal{F}$ (see e.g. \cite{Jordi-Stevo}, \cite{Castillo-Gonzales}).  The most notable representative here is the Schreier space  \cite{Schreier}, 
	$X_\mathcal{S}$, where $\mathcal{S}$ is so-called \textit{Schreier family}:
\[
\mathcal{S} = \lbrace A \subseteq \omega: \lvert A \rvert \leq \min(A) \rbrace.
\]

%We will say that a quasi-norm is non-degenerated, monotone, lower semicontinuous and nice if it satisfies all the relevant properties of Definition \ref{nicenorm}. The definition of $EXH(\varphi)$ and $FIN(\varp%hi)$ also can be naturally upgraded for
%$\varphi$ being a quasi-norm.% and 

\section{Dual quasi-norms} \label{dual-norms}

Now we are going to introduce the main notion of this article. 

	\begin{equation} \label{Upper_norm}
		\lVert x \rVert^{\mathcal{F}} = \inf_{\mathcal{P} \in \mathbb{P}_{\mathcal{F}}} \sum_{F \in \mathcal{P}} \sup_{k \in F} \lvert x(k) \rvert.
	\end{equation}

If $\mathcal{P}$ is a partition of $\omega$, then by $||x||^\mathcal{P}$ we will denote $\| x\|^\mathcal{F}$ where $\mathcal{F}$ is the hereditary closure of $\mathcal{P}$. Note that in this case  
	\[
	\lVert x \rVert^{\mathcal{P}} = \sum_{P \in \mathcal{P}} \sup_{k \in P} \lvert x(k) \rvert.
	\]
	So, for a family $\mathcal{F} \subseteq \mathcal{P}(\omega)$ we have 
	\[ \lVert x \rVert^\mathcal{F} = \inf_{\mathcal{P}\in \mathbb{P}_\mathcal{F}} \lVert x \rVert^\mathcal{P}. \]

Note that in general the expression (\ref{Upper_norm}) does not define a norm. 

\begin{eg}\label{notnorm} Let  $\mathcal{S}$ be the Schreier family. Consider the finitely supported sequences $x = (0,1,1,0,0,0,...)$, $y = (0,0,1,1,1,0,0,0,...)$. We can easily check that $\lVert x
\rVert^{\mathcal{S}} = \lVert y \rVert^{\mathcal{S}} = 1$, but $\lVert x+y \rVert^{\mathcal{S}} = 3$, so the triangle inequality is not satisfied.
\end{eg}

However, it turns out that $\lVert \cdot \rVert^{\mathcal{F}}$ is a quasi-norm, at least if
$\mathcal{F}$ is a compact family. Moreover, it is a nice quasi-norm (in the sense of Definition \ref{nicenorm}).

%Later, we prove that inequality holds for any two sequences and thus it sufficies for $\lVert \cdot \rVert^{\mathcal{F}}$ to be a quasi-norm. \\

Instead of showing that $\lVert \cdot \rVert^\mathcal{F}$ is a quasi-norm and then showing that it is nice, we will do the opposite: first, we will check that $\lVert \cdot \rVert^\mathcal{F}$ satisfies all the conditions of
Definition \ref{nicenorm}. The reason is that lower semicontinuity will allow us to focus on finitely supported sequences.

It is easy to check that if $\mathcal{F}$ is a family covering $\omega$, then $\lVert \cdot \rVert^\mathcal{F}$ is monotone and non-degenerated (in the sense of Definition \ref{nicenorm}). However, it is not necessarily lower semicontinuous. Consider e.g. the
	family $\mathcal{F}$ of all finite subsets of $\omega$ and $x$ defined by $x(k)=1$ for each $k$. Then $\lVert P_n(x) \rVert^\mathcal{F} = 1$ for each $n$ but $\lVert x \rVert^\mathcal{F} = \infty$. We will show that if we additionally assume that
	$\mathcal{F}$ is compact, then $\lVert \cdot \rVert^\mathcal{F}$ is lower semicontinuous and so it is nice. 
	
	\begin{thm} \label{nice_condition}
If $\mathcal{F} \subseteq \mathcal{P}(\omega)$ is compact, hereditary and covering $\omega$, then 
		$\lVert \cdot \rVert^{\mathcal{F}}$ is a nice quasi-norm. 
	\end{thm}

Before we start the proof we recall some definitions and facts about the Vietoris topology. 

Fix a compact $\mathcal{F}\subseteq \mathcal{P}(\omega)$. Every partition can be considered as a subset of $2^{\omega}$ and thus we can treat the set $\mathbb{P}_{\mathcal{F}}$ as a subset of the power set of
	$2^{\omega}$. We can endow this set with the Vietoris topology. 

	\begin{df} 
	Let $X$ be a compact topological space. By $\mathcal{K}(X)$ denote the family of all closed subsets of $X$. The \textit{Vietoris topology} is the one generated by sets of the form 
	\begin{equation} \label{Vietoris}
		\big <U_{1},U_{2},...,U_{n} \big > = \lbrace K \in \mathcal{K}(X): K \subseteq \bigcup_{i \leq n} U_{i} \wedge \forall i \leq n \ K \cap U_{i} \neq \emptyset \rbrace,
	\end{equation} 
	where $U_{i}$ are open subsets of $X$. 
	\end{df} 

	Note that if $X$ is a compact space, then $K(X)$ endowed with the Vietoris topology is compact as well. Also, $K(X)$ is metrizable (by the Hausdorff metric).

	In our case, the role of $X$ is played by $\mathcal{F}$. According to the above, we would like to consider $\mathbb{P}_{\mathcal{F}}$ as a subspace of $\mathcal{K}(\mathcal{F})$. Notice that according to our definition of partition, it contains
	$\emptyset$ and so it is a closed subset of $\mathcal{F}$ (in fact, it forms a sequence converging to $\emptyset$). %However, there is a small technical issue, because no partition
	\begin{lem}
		$\mathbb{P}_{\mathcal{F}}$ is closed in $\mathcal{K}(\mathcal{F})$. Consequently, $\mathbb{P}_{\mathcal{F}}$ is a compact subspace of $\mathcal{K}(\mathcal{F})$. 
	\end{lem}   

\begin{proof}
	Let $\mathcal{G} \in \overline{\mathbb{P}_{\mathcal{F}}}$. We need to prove that $\mathcal{G}$ is a partition, i.e.
	\begin{enumerate}[(i)]
		\item $\emptyset \in \mathcal{G}$,
		\item $\bigcup \mathcal{G} = \omega$,
		\item All elements of $\mathcal{G}$ are pairwise disjoint.
	\end{enumerate}
	Of those (i) is straightforward.

	Suppose now that there exists $n \in \omega$ such that $n \notin \bigcup\mathcal{G}$. Put $U = \lbrace x \in 2^{\omega}: x(n) = 0 \rbrace$. The set $U$ is an open subset of $2^{\omega}$. Take the basic (in Vietoris topology) set $\mathcal{K}_{U}$ = $\lbrace K \in
	\mathcal{K}(\mathcal{F}): K \subseteq U \rbrace$, being an open neighborhood of $\mathcal{G}$. Then $\mathcal{K}_{U} \cap \mathbb{P}_{\mathcal{F}} = \emptyset$. Indeed, otherwise, there would be a partition $\mathcal{P}$ such that $\mathcal{P}
	\subseteq U$, which is impossible, because there is $A \in \mathcal{P}$ such that $n \in A$. The set $\mathcal{K}_{U}$, therefore, testifies that $\mathcal{G} \notin \overline{\mathbb{P}_{\mathcal{F}}}$, which is a contradiction. It proves $(ii)$. 
	\medskip

	To prove (iii) suppose that there are  $A,B \in \mathcal{G}$ such that $A \cap B \neq \emptyset$ and $A \setminus B \ne \emptyset$. Let $n \in A \cap B$ and $m \in A \setminus B$. Consider the following open subsets in $2^{\omega}$
	\[
		U_{1} = \lbrace x \in 2^{\omega}: x(n) = 1 \wedge x(m) = 1 \rbrace \]
		\[U_{2} = \lbrace x \in 2^{\omega}: x(n) = 1 \wedge x(m) = 0 \rbrace \]
		\[ U_{3} = 2^{\omega}, \]
	and the basic set $\big < U_{1}, U_{2}, U_{3} \big>$. Then $\mathcal{G} \cap U_{1} \neq \emptyset$, because $A \in U_{1}$ and $\mathcal{G} \cap U_{2} \neq \emptyset$, since $B \in U_{2}$. So the set $\big <U_{1}, U_{2}, U_{3} \big >$ is an open
	neighborhood of $\mathcal{G}$. If $\mathbb{P}_{\mathcal{F}} \cap \big < U_{1}, U_{2}, U_{3} \big> \neq \emptyset$, then there is a partition $\mathcal{P}$ and sets $K,L \in \mathcal{P}$ such that $n,m \in K$, $n \in L$, and $m \notin L$. But it
	is impossible, since elements of $\mathcal{P}$ are pairwise disjoint. It implies that $\mathbb{P}_{\mathcal{F}} \cap \big <U_{1}, U_{2}, U_{3} \big > = \emptyset$, which is a contradiction.  
	\end{proof}

\begin{proof}[Proof of Theorem \ref{nice_condition}]
As we have mentioned it is enough to show lower semicontinuity. Fix $x \in \mathbb{R}^{\omega}$.

	Assume that $\lVert x \rVert^{\mathcal{F}} = D$ (possibly $D=\infty$). Then for every partition $\mathcal{P}$ we have  $\lVert x \rVert^{\mathcal{P}} \geq D$. For each $n \in \omega$ put $x_{n} = P_{n}(x)$. Suppose that there exists $M < D$ such that
	$\lVert x_{n} \rVert^{\mathcal{F}} < M$ for each $n$. Then for every $n$ there is a partition $\mathcal{P}_{n}$ such that $\lVert x_{n} \rVert^{\mathcal{P}_{n}} < M$. By compactness, we may assume (passing to a subsequence if needed) that
	$(\mathcal{P}_{n})_{n \in \omega}$ converges (in the Vietoris topology) to a partition $\mathcal{P}$. Since $\lVert x \rVert^\mathcal{P} \geq D$, there is $N \in \omega$ such that $\lVert x_{N} \rVert^{\mathcal{P}}\geq D$. There are only finitely many elements  $R_{1},R_{2},...,R_{j}$ of $\mathcal{P}$ having non-empty intersection with $\lbrace 1,2,...,N \rbrace$. For $k \leq j$ put 
	\[ 	U_{k} = \lbrace x \in 2^{\omega}: \forall i \in R_{k} \cap [1,\dots,N] \  x(i) = 1 \rbrace \]
	and consider the basic open set $\big < U_{1},U_{2},...,U_{j}, U_{j+1} \big>$, where $U_{j+1} = 2^{\omega}.$ Then \[ \mathcal{P} \in \big < U_{1},U_{2},...,U_{j},U_{j+1} \big>.\] Indeed, trivially $\mathcal{P} \cap U_{j+1} = \mathcal{P}$
	and for $k \leq j$ we have $ R_k \in \mathcal{P} \cap U_{k} \neq \emptyset$. Since $\mathcal{P}_{n}$ converges to $\mathcal{P}$, there is $k>N$ such that  $\mathcal{P}_{k}
	\in \big < U_{1},U_{2},...,U_{j}, U_{j+1} \big>$. It means that \[ \{P\cap \{1,\dots,N\}\colon P \in \mathcal{P}_k\}  = \{P\cap \{1,\dots,N\} \colon P \in \mathcal{P}\}.\]
	So, \[ \lVert x_k \rVert^{\mathcal{P}_k} \geq \lVert x_N \rVert^{\mathcal{P}_k} = \lVert x_N \rVert^{\mathcal{P}} > M, \]
	a contradiction.
\end{proof}

Now we can prove that $\lVert \cdot \rVert^\mathcal{F}$ is indeed a quasi-norm.

\begin{thm}\label{quasi-quasi} Let $\mathcal{F}$ be a compact hereditary family. Then for every $x,y\in \mathbb{R}^\omega$
	\begin{itemize}
		\item[(a)] if $x,y$ have disjoint supports, then $\lVert x+y \rVert^\mathcal{F} \leq \lVert x \rVert^\mathcal{F} + \lVert y \rVert^\mathcal{F}$,
		\item[(b)] $\lVert x+y \rVert^{\mathcal{F}} \leq 2 (\lVert x \rVert^{\mathcal{F}} + \lVert y \rVert^{\mathcal{F}})$, and so $\lVert \cdot \rVert^\mathcal{F}$ is a quasi-norm.
	\end{itemize}
\end{thm}

\begin{proof}
	Of the above (a) is clear. We will check (b).

	Let $x, y \in \mathbb{R}^\omega$. By lower semicontinuity of $\lVert \cdot \rVert^\mathcal{F}$ it is enough to consider the case when $x$ and $y$ are finitely supported. Moreover, we will assume that $x(k), y(k)\geq 0$ for every $k$ (since $\lVert
	x+y\rVert \leq \lVert |x| + |y| \rVert$ and $\lVert x \rVert = \lVert |x| \rVert$ for each $x,y\in X^\mathcal{F}$) (by $|x|$ we mean a sequence defined by $|x|(k) = |x(k)|$ for each $k \in \omega$).

	Now, let $\mathcal{P}$, $\mathcal{Q} \subseteq \mathcal{F}$ be partitions witnessing $\lVert x\rVert^\mathcal{F}$ and $\lVert y \rVert^\mathcal{F}$ respectively (here we take partitions of the supports of $x$ and $y$). Enumerate $\mathcal{P} = \{P_1, P_2, \dots, P_k\}$ and $\mathcal{Q} = \{ Q_1, Q_2,
	\dots, Q_l\}$. Let $a_i = \max\{x(j)\colon j\in P_i\}$ for $i\leq k$, and $b_i = \max\{y(j)\colon j\in Q_i\}$ for $i\leq l$.
	Re-enumerating $\mathcal{P}$ and $\mathcal{Q}$, if needed, we may assume that $(a_i)$ and $(b_i)$ are non-increasing.

	Now we will define a partition $\mathcal{R}$ of $\mathrm{supp}(x+y)$, intertwining $\mathcal{P}$ and $\mathcal{Q}$ in the following way:
	\[ R_{2n+1} = P_n \setminus \bigcup_{i\leq 2n} R_i \mbox{ for }0\leq n \leq k, \mbox{ and } R_{2n} = Q_n \setminus \bigcup_{i < 2n} R_i \mbox{ for } 1\leq n\leq l. \]
	Then $\max\{x(i)+y(i)\colon i\in R_1\} \leq a_1+b_1$, $\max\{x(i)+y(i)\colon i\in R_2\} \leq a_2+b_1$ and so on. Therefore, the partition $\mathcal{R}$ witnesses that 
	\begin{multline}
	\lVert x+ y \rVert^\mathcal{F} \leq (a_1+b_1) + (b_1 + a_2) + (a_2 + b_2) + (b_2+a_3) + \dots \leq \\ \leq a_1 + 2 (a_2 + \dots + a_k) +  2 (b_1 + \dots + b_l) \leq  2(\lVert x\rVert^\mathcal{F} + \lVert y\rVert^\mathcal{F}). 
	\end{multline}
\end{proof}

%For an extended quasi-norm $\varphi: \mathbb{R}^{\omega} \rightarrow [0,\infty]$ define
%	\begin{center}
%		$FIN(\varphi) = \lbrace x \in \mathbb{R}^{\omega}: \varphi(x) < \infty \rbrace,$\\
%		$EXH(\varphi) = \lbrace x \in \mathbb{R}^{\omega}: \lim \limits_{n \rightarrow \infty} \varphi(P_{\omega \setminus n}(x)) = 0 \rbrace. $
%	\end{center}
Below we prove a natural counterpart of Theorem \ref{Banach} (mimicking the proof from \cite{NaFa}).

\begin{thm}\label{quasi-Banach} If $\mathcal{F}$ is a compact hereditary family and $\mathcal{F}$ covers $\omega$, then $FIN(\lVert \cdot \rVert^\mathcal{F})$ and $EXH(\lVert \cdot \rVert^\mathcal{F})$ are quasi-Banach spaces.
\end{thm}

\begin{proof} 
That $\lVert \cdot \rVert^\mathcal{F}$ is a nice quasi-norm follows directly from Theorem \ref{quasi-quasi} and from Theorem \ref{nice_condition}. 
	
We are going to show that $FIN(\lVert \cdot \rVert^\mathcal{F})$ is complete and then that $X^\mathcal{F}$ is a closed subset of $FIN(\lVert \cdot \rVert^\mathcal{F})$.

For simplicity denote $\varphi  = \lVert \cdot \rVert^\mathcal{F}$. 

%Fix $c$ such that $\varphi(x+y)\leq c(\varphi(x) + \varphi(y))$.

	%As $\varphi$ is finite on $c_{00}$ and $\varphi(x)\leq c (\varphi(P_n(x))+\varphi(P_{\omega\setminus n}(x)))$, $EXH(\varphi)\subseteq FIN(\varphi)$ follows.

First we will prove that $FIN(\varphi)$ is complete. Let $(x_n)$ be a Cauchy sequence in $FIN(\varphi)$. Applying monotonicity,  $\varphi(P_{\{k\}}(x_n-x_m))\leq \varphi(x_n-x_m)$ for every $k,n,m$, and hence $(P_{\{k\}}(x_n))_{k\in\omega}$ is a Cauchy sequence in the $k$th $1$-dimensional
coordinate space of $\mathbb{R}^\omega$ (which is a quasi-Banach space, as $\varphi$ is finite on $c_{00}$), $P_{\{k\}}(x_n)\xrightarrow{n\to\infty}y_k$ for some $y_k$. Put $y=(y_k)$. We will first show that $y\in FIN(\varphi)$. The sequence
	$\{x_n\colon n\in\omega\}$ is bounded, let
say $\varphi(x_n)\leq B$ for every $n$. We show that $\varphi(y)\leq 4B$, i.e. (by the lower semicontinuity of $\varphi$) $\varphi(P_M(y))\leq 4B$ for every $M\in\omega$. Fix an $M>0$. If $n$ is large enough, say $n\geq n_0$, then
	$\varphi(P_{\{k\}}(y-x_n))\leq \dfrac{B}{M}$ for every $k<M$ and hence \[\varphi(P_M(y))\leq 2(\varphi(P_M(y-x_n))+\varphi(P_M(x_n))) \leq 2(\sum_{k<M}\varphi(P_{\{k\}}(y-x_n))+\varphi(x_n))\leq 4B.\]
The first inequality follows from Theorem \ref{quasi-quasi}(c) and the second from Theorem \ref{quasi-quasi}(a).

Now we will prove that $x_n\to y$. If not, then there are $\varepsilon>0$ and $n_0<n_1<\dots<n_j<\dots$ such that $\varphi(x_{n_j}-y)>\varepsilon$, that is, $\varphi(P_{M_j}(x_{n_j}-y))>\varepsilon$ for some $M_j\in \omega\setminus\{0\}$ for every $j$. Pick
	$j_0$ such that  $\varphi(x_{n_{j_0}}-x_n)< \dfrac{\varepsilon}{2}$ for every $n\geq n_{j_0}$ and then pick $j_1>j_0$ such that $\varphi(P_{\{k\}}(x_{n_{j_1}}-y))\leq \dfrac{\varepsilon}{2 M_{j_0}}$ for every $k<M_{j_0}$. Then, using Theorem \ref{quasi-quasi}(a)
\[
	\varepsilon<\varphi(P_{M_{j_0}}(x_{n_{j_0}}-y))\leq \varphi(P_{M_{j_0}}(x_{n_{j_0}}-x_{n_{j_1}}))+\sum_{k<M_{j_0}}\varphi(P_{\{k\}}(x_{n_{j_1}}-y)) < \varepsilon,\]
a contradiction.

\smallskip
Now we will show that $EXH(\varphi)=\overline{c_{00}}$. The space $c_{00}$ is dense in $EXH(\varphi)$ because $\varphi(x-P_n(x))=\varphi(P_{\omega\setminus n}(x))\xrightarrow{n\to\infty}0$ for every $x\in EXH(\varphi)$. We have to show that $EXH(\varphi)$ is closed. Let
	$x\in FIN(\varphi)$ be an accumulation point of $EXH(\varphi)$. For any $\varepsilon>0$ we can find  $y\in EXH(\varphi)$ such that $\varphi(x-y)<\varepsilon$, and then $n_0$ such that $\varphi(P_{\omega\setminus n}(y))<\varepsilon$ for every
	$n\geq n_0$. If $n\geq n_0$ then $\varphi(P_{\omega\setminus n}(x))\leq 2(\varphi(P_{\omega\setminus n}(x-y))+\varphi(P_{\omega\setminus n}(y)))< 4\varepsilon$.
%\smallskip
%$(e_n)$ is a $1$-unconditional basis in $\mrm{EXH}(\Phi)$: This follows from monotonicity of $\Phi$. 
%\smallskip
%Suppose $(b_n)$ is a $1$-unconditional basis in a Banach space  $X$. In particular, for each $x\in X$ there is a unique $(x_n)\in \mbb{R}^\om$ such that  $x = \sum_{n\in\om} x_n b_n$, and hence we can assume that $X\subseteq\mbb{R}^\om$. For
%$x=(x_n)\in\mbb{R}^\om$ define $\Phi(x)=\sup\{\|\sum_{k<n}x_kb_k\|\colon n\in\om\}$. Applying $1$-unconditionality of $(b_n)$, $\Phi$ is a nice norm, and of course $\Phi(x)=\|x\|$ for every $x\in X$. Now, $x\in\mrm{EXH}(\Phi)$ iff
%$\Phi(P_{\om\setminus n}(x))\xrightarrow{n\to\infty} 0$ iff $\sum_{n\in\om} x_nb_n$ is Cauchy/convergent iff $x\in X$.
\end{proof}

We will be mainly interested in $EXH(\lVert \cdot \rVert^\mathcal{F})$ and so it will be convenient to denote $X^\mathcal{F} = EXH(\lVert \cdot \rVert^\mathcal{F})$. The main corollary of this section is the following reformulation of (a part of)
Theorem \ref{quasi-Banach}:

\begin{thm} If $\mathcal{F}$ is compact, hereditary and covering $\omega$, then $X^\mathcal{F}$ is a quasi-Banach space.
\end{thm}

The following is a simple consequence of (a) of Theorem \ref{quasi-quasi}.

\begin{cor} If a family $\mathcal{F}$ is a hereditary closure of a partition $\mathcal{P}$, then the formula (\ref{Upper_norm}) defines a norm. 
\end{cor}

As we already know, in general, the formula (\ref{Upper_norm}) does not need to define a norm, but we can consider \emph{the Banach envelope} of  $X^\mathcal{F}$.%always 'convexify' it to a norm.

\begin{df} \label{norma}
Let $\mathcal{F}$ be a compact, hereditary family. Let
\begin{equation} \label{equi_norm}
	\left \vvvert x \right \vvvert^\mathcal{F} = \inf \bigg \{ \sum \limits_{i=1}^{n} \lVert x_i \rVert^{\mathcal{F}} : n \in \omega, x_1,...,x_n \in X, x = \sum \limits_{i=1}^{n} x_i \bigg \}.
\end{equation}
Since $\lVert \cdot \rVert^\mathcal{F}$ is a quasi-norm, this formula defines a norm.  The space  $\widehat{X}^\mathcal{F} = EXH(\vvvert \cdot \vvvert^\mathcal{F})$ is the Banach envelope of $X^\mathcal{F}$ (see \cite{Kalton-F}).
\end{df}

\begin{rem} \label{normability} Clearly, for every compact, hereditary family $\mathcal{F}$ and for $x\in c_{00}$ we have
\[
	\left \vvvert x \right \vvvert^\mathcal{F} \leq \lVert x \rVert^{\mathcal{F}}.
\]
	If there is $C>0$ such that for each sequence $(x_i)$ of vectors in $c_{00}$
\begin{equation} \label{1-convex}
\Big \lVert \sum \limits_{i=1}^{n} x_i \Big \rVert^{\mathcal{F}} \leq C \sum \limits_{i=1}^{n} \lVert x_i \rVert^{\mathcal{F}},
\end{equation}
then 
\[
	C \left \vvvert x \right \vvvert^\mathcal{F} \geq \lVert x \rVert^{\mathcal{F}}.
\]
	Property (\ref{1-convex}) is called \emph{1-convexity} and it is equivalent to the normability of a quasi-Banach space (by $\left \vvvert \cdot \right \vvvert^\mathcal{F}$).
In the next section we will show that in general $\lVert \cdot \rVert^\mathcal{F}$ does not have to be 1-convex.

Also, we will prove that the Banach space induced by $\left \vvvert \cdot \right \vvvert^\mathcal{F}$ is  isomorphic to $X^*_\mathcal{F}$.
\end{rem}

In what follows we will be mainly interested in the families of finite subsets of $\omega$. In this case $FIN(\lVert \cdot \rVert^\mathcal{F}) = EXH(\lVert \cdot \rVert^\mathcal{F})$:

\begin{prop}\label{exh=fin} If $\mathcal{F} \subseteq [\omega]^{<\omega}$ is a compact hereditary family covering $\omega$, then \[ X^\mathcal{F} = FIN(\lVert \cdot \rVert^\mathcal{F}) \]
\end{prop}
\begin{proof}
For each $x \in \mathbb{R}^{\omega}$ and for fixed $n$ we can write $x = x_{n} + x'_{n}$, where $x_{n} = P_{n}(x)$ and $x'_{n} = P_{\omega \setminus n}(x)$. Thus if $x \in X^{\mathcal{F}}$ then $\lVert x'_{n} \rVert^{\mathcal{F}} \rightarrow 0$ and 
\[ \lVert x \rVert^{\mathcal{F}} \leq 2 (\lVert x_{n} \rVert^{\mathcal{F}} + \lVert x'_{n} \rVert^{\mathcal{F}}) < \infty, \]
because $x_{n}$ is finitely supported. It shows that $X^{\mathcal{F}} \subseteq FIN(\lVert \cdot \rVert^{\mathcal{F}})$. On the other hand, if $x \in FIN(\lVert \cdot \rVert^{\mathcal{F}})$, then there is a partition $\mathcal{G} = \lbrace G_{k}: k
\in \omega \rbrace \subseteq \mathcal{F}$ such that $\mathlarger{\sum}_{k \in \omega} \sup \limits_{j \in G_{k}} \lvert x(j) \rvert < \infty$. It implies that 
\[\mathlarger{\sum}_{k \geq n} \sup \limits_{j \in G_{k}} \lvert x(j) \rvert \xrightarrow{n \rightarrow \infty} 0. \]

Let $\varepsilon>0$ and fix $m$ such that $\mathlarger{\sum}_{k\geq m} \sup \limits_{j\in G_k} |x(j)| < \varepsilon$.
Let $n > \max (\bigcup \limits_{i<m} G_i)$. Then
\[ \lVert x'_n \rVert^\mathcal{F} \leq \lVert x'_n \rVert^\mathcal{G} \leq \sum_{k\geq m} \sup_{j\in G_k} |x(j)| < \varepsilon. \]
It finishes the proof. 
\end{proof}

\section{$X^\mathcal{F}$ and the dual of $X_\mathcal{F}$}\label{how-close}

In this section we will examine how close is $X^\mathcal{F}$ to $X^*_\mathcal{F}$, the space dual to $X_\mathcal{F}$.

In case $\mathcal{F}$ is simple enough (i.e. it is generated by a partition), it is not hard to see that $X^\mathcal{F}$ is isometrically isomorphic
to $X^*_\mathcal{F}$ (Proposition \ref{for-partitions}). In general case, this is not true. However, $X^*_\mathcal{F}$ is always the Banach envelope of $X^\mathcal{F}$.
%one can say that the spaces $X^\mathcal{F}$ and $X^*_\mathcal{F}$ have the same geometric core: the convexification of $X^\mathcal{F}$, the space $\widehat{X}^\mathcal{F}$, is isometrically isomorphic to $X^*_\mathcal{F}$. In fact, they are so much similar that the only distinguishing property we have managed to find is the fact that $X^*_\mathcal{F}$ is always a Banach space whereas $X^\mathcal{F}$ typically is not.

First of all, we show that for some particular families $\mathcal{F}$ spaces $X^{\mathcal{F}}$ and $X^*_{\mathcal{F}}$ are indeed the same.

\begin{prop} \label{for-partitions}
		Suppose $\mathcal{P}$ is a partition of $\omega$ (into finite sets) and $\mathcal{F}$ is its hereditary closure. Then $X_\mathcal{F}^*$ is isometrically isomorphic to $X^\mathcal{F}$.
	\end{prop}
	\begin{proof}
Enumerate $\mathcal{P} = \{F_1, F_2, \dots\}$. It is known that for $\mathcal{F}$ being generated by partition, $X_\mathcal{F}$ is isometrically isomorphic to $\bigoplus_{c_0} \ell_{1}^{|F_{n}|}$ and so its dual space is isometrically isomorphic to $\bigoplus_{\ell_{1}}c_{0}^{|F_n|}$. \\
Let $y \in X^{\mathcal{F}}$. Then $\|y\|^{\mathcal{F}} = \sum \limits_{n \in \omega} \max \limits_{k \in F_n} |y(k)|$. Taking $y_n = P_{F_{n}}(y)$ for each $n$, we can see $y_n$ as element of $\mathbb{R}^{|F_n|}$. Thus $\|y\|^{\mathcal{F}} = \sum \limits_{n \in \omega} \|y_{n}\|_{\infty}$, which gives us the norm on $\bigoplus_{\ell_{1}}c_{0}^{|F_n|}$.

	\end{proof}

Let $\mathcal{F}$ be a compact, hereditary family covering $\omega$. Define $T\colon c_{00} \rightarrow X^*_{\mathcal{F}}$, a linear operator given by
\begin{equation} \label{identity_operator}
T(y)(x) = \sum \limits_{k \in \omega} x(k)y(k)
\end{equation}
for $x \in X_{\mathcal{F}}$. It is plain to check that $T$ is injective. Also, let $T_0 \colon c_{00} (\lVert \cdot \rVert^{\mathcal{F}}) \to X^*_{\mathcal{F}}$ and $T_1 \colon c_{00} (\vvvert \cdot \vvvert^{\mathcal{F}}) \to
	X^*_{\mathcal{F}}$ denote the operators given by the same formula as $T$. 

\begin{prop}\label{identityoperator} %The mapping $T$ defined above is injective. Also, both $T_0 \colon c_{00} (\lVert \cdot \rVert^{\mathcal{F}}) \to X^*_{\mathcal{F}}$ and $T_1 \colon c_{00} (\vvvert \cdot \vvvert^{\mathcal{F}}) \to
	%X^*_{\mathcal{F}}$ are continuous. More precisely, for every $y\in c_{00}$ we have \[ \lVert T(y) \rVert^*_\mathcal{F} \leq \vvvert y \vvvert^\mathcal{F} (\leq \lVert y \rVert^\mathcal{F}). \]

$T_0$ and $T_1$ are continuous with the norm 1.
	
\end{prop}

\begin{proof}

To prove that $T_0$ is continuous, take finitely supported $y$ and let $\mathcal{P}$ be such that $\|y\|^{\mathcal{F}} = \sum \limits_{F \in \mathcal{P}} \max \limits_{k \in F} |y(k)|$. Then for every $x \in X_\mathcal{F}$ with $\|x\|_{\mathcal{F}} \leq 1$ we have
\[
\big| \sum \limits_{k \in \omega} x(k)y(k) \big| = \big| \sum \limits_{F \in \mathcal{P}} \sum \limits_{k \in F} x(k)y(k) \big| \leq \sum \limits_{F \in \mathcal{P}} \max\limits_{k \in F}|y(k)| \sum \limits_{k \in F} |x(k)| \leq \|y\|^{\mathcal{F}}.
\]
Thus 
\begin{equation} \label{T_0 inequality}
\|T(y)\|^*_\mathcal{F} \leq \|y\|^{\mathcal{F}},
\end{equation}
and so $T_0$ is continuous. \\
To show that $T_1$ is continuous we use (\ref{T_0 inequality}). Notice that for $y = \sum \limits_{i\leq n} y_i$ we have
\[ \lVert T(y) \rVert^*_\mathcal{F} \leq \sum_{i\leq n} \lVert T(y_i) \rVert^*_\mathcal{F} \leq \sum_{i\leq n} \lVert y_i \rVert^\mathcal{F}. \]
It implies that $\|T(y)\|^*_\mathcal{F} \leq \vvvert y \vvvert^\mathcal{F}$, hence $T_1$ is continuous. 
\end{proof}
Note that by above proposition, as $X_\mathcal{F}^*$ is complete, we can extend the operator $T_0$ to a continuous injective linear operator $X^{\mathcal{F}}\to X^*_{\mathcal{F}}$, denoted also by $T_0$. The same holds true for $T_1$ and $\widehat{X}^\mathcal{F}$.
%(\ref{identity_operator}) to continuous injective linear operator between $X^{\mathcal{F}}$ and $X^*_{\mathcal{F}}$. 

Now we are ready to prove the final theorem. 

\begin{thm} \label{final_isomorphism}
	Let $\mathcal{F} \subseteq [\omega]^{< \omega}$ be a compact, hereditary family covering $\omega$. Then $\widehat{X}^{\mathcal{F}}$ is isometrically isomorphic to $X^{*}_{\mathcal{F}}$.   
\end{thm}

In the proof we will use some general facts about the spaces of the form $X^*_\mathcal{F}$ and their extreme points.

\begin{df}
Let $K$ be a subset of a vector space $X$. We say that $e \in K$ is an \emph{extreme point} of $K$ if there do not exist $x,y \in K$ (different than $e$) and $t \in (0,1)$ such that $e = (1-t)x + ty$. If $e$ is an extreme point of $K$ then the following condition is satisfied
\begin{equation} \label{extreme}
e+x, e-x \in K \Rightarrow x = 0
\end{equation}
The other implication is true if we additionally assume that $K$ is convex. 
The set of all extreme points of $K$ is denoted by $E(K)$. However, in the context of Banach spaces, we abuse slightly this notation - for a (quasi) Banach space $X$ $E(X)$ denotes the set of all extreme points of the unit ball $B_X$.   
\end{df}

One of the few properties of the spaces dual to combinatorial spaces which are exposed in the literature is that their balls can be well approximated by extreme points.

\begin{df} \label{CSRP}
We say that a quasi-Banach space $X$ has \emph{convex series representation property (CSRP)} if for every $x \in B_{X}$ there exists a sequence $(\lambda_{n})$ of positive real numbers with $\sum \limits_{n \in \omega} \lambda_{n} = 1$ and a sequence $(u_{n})$ of extreme points of $B_{X}$ such that
\begin{equation} \label{CSRP_eq}
x = \sum \limits_{n \in \omega} \lambda_{n} u_{n}.
\end{equation}
\end{df}

Although the notion of an extreme point is usually considered in the context of convex sets, the definition itself does not require a priori convexity. Thus we can consider extreme points in the case of non-convex sets as well. 

%In the case of (quasi)-Banach spaces $X$ it is common to denote by $E(X)$ the set of all extreme points of the unit ball of $X$ and we will follow this  notation.

The combinatorial spaces and their duals were studied geometrically in the context of extreme points.
Note that if $\mathcal{F}$ is a compact, hereditary family of finite sets, then all the extreme points of the unit ball of $X^*_{\mathcal{F}}$ are finitely supported and there are only finitely many extreme points with a given support (see
\cite{Antunes},  \cite{Brech}). It is known (see \cite{Antunes}) that $X^{*}_{\mathcal{F}}$ has CSRP, for $\mathcal{F}$ as above. 

In \cite{Antunes} the authors provide proof of Gowers's theorem regarding the characterization of extreme points of the unit ball in $X^{*}_{\mathcal{F}}$. In his blog \cite{Gowers09}, Gowers states (without proof) that the set of extreme points is of the form

\begin{equation} \label{dual_extreme}
\Big \lbrace \sum \limits_{i \in F} \varepsilon_i e^{*}_{i}: F \in \mathcal{F}^{MAX}, \varepsilon_{i} \in \lbrace -1,1 \rbrace \Big \rbrace
\end{equation}

where 
\begin{itemize}
\item $e^{*}_{i}$ are functionals given by $e^{*}_{i}(e_{j}) = 1$ if $i=j$ and $e^{*}_{i}(e_{j}) = 0$ otherwise for the canonical Schauder basis $(e_{i})$, $i,j \in \omega$.  
\item $\mathcal{F}^{MAX}$ is a family of \emph{maximal} sets from $\mathcal{F}$, i.e. these sets $F$ for which $F \cup \lbrace k \rbrace \notin \mathcal{F}$ for every $k \in \omega$. 
\end{itemize}

Actually, the fact that $E(X^{*}_{\mathcal{F}})$ is given by (\ref{dual_extreme}) was proven only for Schreier space and for \emph{higher order Schreier spaces} for the definition see \cite{Alspach-Argyros}. However, that result holds also for general compact, hereditary family
$\mathcal{F}\subseteq [\omega]^{<\omega}$ (see Remark 4.4 from \cite{Antunes} and Proposition 5 from \cite{Brech}). 
We will show that $X^{\mathcal{F}}$ has basically \emph{the same} extreme points, that is $T_0(E(X^\mathcal{F}))=E(X^*_\mathcal{F})$: %(interpreting elements of $X^\mathcal{F}$ as elements of $X^*_\mathcal{F}$ in the sense of Proposition \ref{identityoperator}): %The lemma \ref{extreme_to_extreme} explains what it means.

\begin{prop} \label{extreme_to_extreme}
	Assume that $\mathcal{F}\subseteq [\omega]^{<\omega}$ is a compact, hereditary family covering $\omega$.
	A vector $y\in X^\mathcal{F}$ is an extreme point of the unit ball of $X^\mathcal{F}$ if and only if it is of the form
\begin{equation} \label{seq_1}
y(i) = 
\begin{cases}
\varepsilon_{i}, & \text{if } i \in F\\
0 & \text{otherwise,}
\end{cases}
\end{equation}
	for some $F\in \mathcal{F}^{MAX}$ and $\varepsilon_i\in \{-1,1\}$.
\end{prop}

\begin{proof}
%Consider $f \in E(X^{*}_{\mathcal{F}})$, i.e. $f = \sum \limits_{i \in F} \varepsilon_{i} e^{*}_{i}$ for some maximal $F$. For every $x \in X_{\mathcal{F}}$ we have $e^{*}_{i}(x) = x(i)$, so $f(x) = \sum \limits_{i \in F} \varepsilon_{i}x(i)$. Thus if we define sequence $y$ by
%\begin{equation} \label{seq_1}
%y(i) = 
%\begin{cases}
%\varepsilon_{i}, & \text{if } i \in F\\
%0 & \text{otherwise}
%\end{cases}
%\end{equation}

%then $T(y) = f$. Note that this $y$ can be written as $\sum \limits_{i \in F}\varepsilon_{i}e_{i}$ for a standard Schauder basis $(e_i)$. \\ It is also clear that for such $y$ $T(y)$ is an extreme point of the unit ball in $X^{*}_{\mathcal{F}}$. Thus if we denote 
%$$
%E = \Big \lbrace \sum \limits_{i \in F} \varepsilon_i e_{i}: F \in \mathcal{F}^{MAX} \Big \rbrace
%$$

%then we obtain $T[E] = E(X^{*}_{\mathcal{F}})$. It remains to show that $E = E(X^{\mathcal{F}})$. \\

First, assume that $y$ equals  1, up to an absolute value, on some maximal set $F \in \mathcal{F}$. %Now suppose that for $u \in X^{\mathcal{F}}$ we have $y+u, y-u \in B_{X^{\mathcal{F}}}$. Then for every $k \in supp(y) \cap supp(u)$ we have $|1 \pm
	%u(k)| \leq 1$ or $|-1 \pm u(k)| \leq 1$. In both cases it implies that $u(k) = 0$, hence $supp(y) \cap supp(u) = \emptyset$. Thus $y(m) + u(m) = \pm 1$ for some $m \in \omega$. In particular, by maximality of $F$, $\lVert y \pm u \rVert^{\mathcal{F}} \geq 1$, and thus $u = 0$.

Suppose $y$ is not extreme and $y = (1-t)x + tz$ for some $0 < t <1$ and $x,z \in B_{X^{\mathcal{F}}}$. In particular, absolute values of $x$ and $z$ do not exceed $1$. Suppose that e.g. $|x(k)| < 1$ for some $k \in F$. Then $1 = |y(k)| \leq (1-t)|x(k)| + t|z(k)| < 1$, which is a contradiction. So $|x(k)| = |z(k)| = 1$ for each $k \in F$. On the other hand, if $x(k) \neq 0$ for $k \notin F$, then by maximality of $F$ it follows that $\lVert x \rVert^{\mathcal{F}} > 1$. It implies that $x(k) = 0$ for every $k \notin F$. Hence $x = y$. This is a contradiction and so $y$ must be an extreme point.  

Now suppose that $y \in E(X^{\mathcal{F}})$. Then $\lVert y \rVert^{\mathcal{F}} = 1$. %Suppose that $y$ is not of the form as in the proposition. 
%	We will consider two cases (clearly it suffices to assume that $y \in c_{00}$).
	Let $\mathcal{P} \in \mathbb{P}_{\mathcal{F}}$ for which $\lVert y \rVert^{\mathcal{F}} = \lVert y \rVert^{\mathcal{P}}$. Notice that for every $P\in \mathcal{P}$ we have $|y(i)| = |y(j)|$ for every $i,j\in P$. Suppose otherwise. Then there is
	$P\in \mathcal{P}$ and $i,j \in P$ such that $|y(j)|<|y(i)|$ and so for $\eta< |y(i)|-|y(j)|$ we would have $\lVert y \pm \eta e_j \rVert^\mathcal{F} \leq \lVert y \pm \eta e_j \rVert^\mathcal{P} \leq 1$, hence $y$ would not be an extreme point.

	It follows, that if $\mathrm{supp}(y)\in \mathcal{F}$, then $y$ needs to be of the promised form (in particular $\mathrm{supp}(y)$ is a maximal set in $\mathcal{F}$, otherwise $\lVert y\pm e_i\lVert=1$ for $i\not\in\mathrm{supp}(y)$ with $\mathrm{supp}(y)\cup\{i\}\in\mathcal{F}$). If $\mathrm{supp}(y)\notin \mathcal{F}$, then we may find distinct $P_0, P_1\in \mathcal{P}$ and $a_0, a_1 \ne 0$ such that $y(i)=a_j$ for $i\in P_j$, $j\in \{0,1\}$. Since
	$\lVert y \rVert^\mathcal{F} = 1$, $|a_0|, |a_1|<1$. But then for sufficiently small $\eta>0$ and for $u\in B_{X^\mathcal{F}}$ defined by 
\begin{equation} 
u(i) = 
\begin{cases}
\eta, & \text{if } i \in P_0\\
	- \eta, & \text{if } i \in P_1\\
0 & \text{otherwise,}
\end{cases}
\end{equation}
we would have \[ \lVert y \pm u \rVert^\mathcal{F} \leq \lVert y \pm u \rVert^\mathcal{P} = \lVert y\rVert^\mathcal{P} = 1. \] 

So, $y$ has to be of the form as in the proposition.
\end{proof}

We shall need also the following property of the basis of the space $X_\mathcal{F}$ which can be shown by repeating directly the proof of Proposition 3.10 of \cite{BirdLaustsen}. 

\begin{prop}\label{shrinking} Let $\mathcal{F} \subseteq [\omega]^{< \omega}$ be a compact, hereditary family covering $\omega$. Then the unit vector basis of $X_\mathcal{F}$ is shrinking, thus the biorthogonal basic sequence $(e_i^*)_{i=1}^\infty\subset X_\mathcal{F}^*$ is a basis for $X_\mathcal{F}^*$. 
\end{prop}

\begin{proof}[Proof of Theorem \ref{final_isomorphism}]
We shall use the natural identification of $X_\mathcal{F}^*$  with a subspace of $\mathbb{R}^\omega$ by the map $X_\mathcal{F}^*\ni f\mapsto (f(e_n))_{n=1}^\infty\in\mathbb{R}^\omega$. In this setting the extended mapping $T_1$ of Proposition \ref{identityoperator} (see the remark after Proposition \ref{identityoperator}) becomes the formal inclusion $\widehat{X}^\mathcal{F}\hookrightarrow X_\mathcal{F}^*$, and, by Proposition \ref{shrinking}, $X^*_\mathcal{F}$ is the completion of $(c_{00},\|\cdot\|^*_\mathcal{F})$. % (here we use the same notation for a dual norm as in Proposition \ref{for-partitions}).  

By Proposition \ref{identityoperator} we have $\lVert \cdot \rVert^{*}_\mathcal{F} \leq \vvvert \cdot \vvvert^\mathcal{F}$  on $\widehat{X}^\mathcal{F}$.  We will prove that $\vvvert\cdot\vvvert^\mathcal{F}\leq \lVert\cdot\rVert^*_\mathcal{F}$ on $c_{00}$, which implies equality of $\vvvert\cdot\vvvert^\mathcal{F}$ and $\|\cdot\|^*_\mathcal{F}$ on $c_{00}$. Then Proposition \ref{shrinking} and the definition of $\widehat{X}^\mathcal{F}$ yield $\widehat{X}^\mathcal{F}=X_\mathcal{F}^*$ and equality of $\vvvert\cdot\vvvert^\mathcal{F}$ and $\lVert\cdot\rVert^*_\mathcal{F}$ on 	$\widehat{X}^\mathcal{F}=X_\mathcal{F}^*$.  
	 
%Let $B_{0}, B_{1}$ and $B^*$ denote the unit ball in quasi-norm $\lVert \cdot \rVert^{\mathcal{F}}$, the unit ball in norm $\vvvert \cdot \vvvert$, and the  dual unit ball, respectively. 
	
We will prove that $\vvvert \cdot \vvvert^\mathcal{F} \leq \| \cdot \|^{*}_\mathcal{F}$ on $c_{00}$ by showing that $B_{X^*_\mathcal{F}}\cap c_{00} \subseteq B_{\widehat{X}^\mathcal{F}}$.
	
Fix finitely supported $x \in B_{X^*_\mathcal{F}}$ and let $A = \mathrm{supp}(x)$. Since $X^*_{\mathcal{F}}$ has CSRP, we have $x = \sum \limits_{k \in \omega} \lambda_k u_k$ for $u_k \in E(X^*_{\mathcal{F}})$ and $\lambda_k$ such that $\sum \limits_{k \in \omega} \lambda_k = 1$. 

By continuity of $P_A$ we have   $x = \sum \limits_{k \in \omega} \lambda_k P_A(u_k)$. By the form of extreme points (see \eqref{dual_extreme}) the set $\{P_{A}(u_k)\colon k\in \omega\}$ is finite and so we may enumerate it as $\{v_i\colon i \leq n\}$ for some $n \in \omega$. Also, there are $\alpha_i > 0$, $i\leq n$, such that $\sum \limits_{i=1}^{n}\alpha_i = 1$ and $x = \sum \limits_{i=1}^{n}\alpha_i v_i$. It means that $x \in \mathrm{conv}(P_A[E(B_{X^*_\mathcal{F}})])$, where $\mathrm{conv}(K)$ denotes the convex hull of a set $K$. Since each $u_i$ is an extreme point, we have that $\varepsilon_{0}P_A(u_i) + \varepsilon_{1}P_{\omega \setminus A}(u_i)$ is an extreme point for $\varepsilon_{0}, \varepsilon_{1} \in \{-1,1\}$. In particular, as $v_i = \dfrac{1}{2} \Big( u_i + (P_A(u_i) - P_{\omega \setminus A}(u_i)\Big)$, we have $x \in \mathrm{conv}(E(X^*_\mathcal{F}))$ and thus $B_{X^*_\mathcal{F}}\cap c_{00} = \mathrm{conv}(E(X^*_\mathcal{F}))$.

On the other hand, by Proposition \ref{extreme_to_extreme} we know that $E(X^*_\mathcal{F}) \subseteq B_{X^\mathcal{F}} \subseteq B_{\widehat{X}^\mathcal{F}}$ and, since $B_{\widehat{X}^\mathcal{F}}$ is convex, we obtain that $B_{X^*_\mathcal{F}} \cap c_{00}\subseteq B_{\widehat{X}^\mathcal{F}}$. 
\end{proof}

Now we will show a result which indicates that the connection between $X^\mathcal{F}$ and $X^*_\mathcal{F}$ is quite strong. For each compact family $\mathcal{F}$ the space $X^\mathcal{F}$ is a (quasi-Banach) pre-dual of
	$(X_{\mathcal{F}})^{**}$. In other words $X^\mathcal{F}$ and $X^*_\mathcal{F}$ have the isometrically isomorphic dual spaces. In fact, this is a direct corollary of Theorem \ref{final_isomorphism} and \cite[Chapter 2.4]{Kalton-F}. We enclose a detailed proof.

\begin{thm} \label{isomorphism}
		If $\mathcal{F}\subseteq [\omega]^{<\omega}$ is a compact hereditary family covering $\omega$, then $(X^{\mathcal{F}})^{*}$ is isometrically isomorphic to $(X_{\mathcal{F}})^{**}$.
	\end{thm}

\begin{proof}
By Theorem \ref{5.5}  (and the fact that $X_\mathcal{F}$ does not contain an isomorphic copy of $\ell_1$ if $\mathcal{F}$ is as above) the space $(X_{\mathcal{F}})^{**}$ is isometrically isomorphic to $FIN(\lVert \cdot \rVert_{\mathcal{F}})$. As in the proof of Proposition \ref{shrinking} we use the natural identification of $(X^\mathcal{F})^*$  with a subspace of $\mathbb{R}^\omega$ via the map $(X^\mathcal{F})^*\ni f\mapsto (f(e_n))_{n=1}^\infty\in\mathbb{R}^\omega$.
We need to prove that $\lVert y \rVert_{*}^{\mathcal{F}} = \lVert y \rVert_{\mathcal{F}}$ for any $y\in\mathbb R^\omega$, where $\lVert \cdot \rVert_{*}^{\mathcal{F}}$ denotes the functional norm on $X^{\mathcal{F}}$. 

Take any $y \in \mathbb R^\omega$. %{\color{red} Cbyba nie wystarczy badać równości norm na $c_{00}$, tylko na całym $\mathrm{R}^\omega$? następne szacowania idą, tylko jako nierówności, bo niekoniecznie np. supremum w def. normy będzie osiągane}. 
For any set $F_{0} \in \mathcal{F}$ %such that $\lVert y \rVert_{\mathcal{F}} = \mathlarger{\sum}_{n \in F_{0}} \lvert y(n) \rvert$. 
consider $x_{0} \in X^\mathcal{F}$ given by
	\[
	x_{0}(n) = 
	\begin{cases}
		sgn(y(n)), & \text{if } n \in F_{0},\\
		0, & \text{otherwise.}
	\end{cases}
	\]
This is a vector of norm at most 1 in $X^{\mathcal{F}}$ and thus 
	\[
	\lVert y \rVert_{*}^{\mathcal{F}} \geq \lvert \sum_{n \in F_{0}} x_{0}(n)y(n) \rvert = \sum_{n \in F_{0}} \lvert y(n) \rvert %= \lVert y \rVert_{\mathcal{F}}. 
	\]
 As $F_0\in\mathcal{F}$ was arbitrary, we obtain $\lVert y\rVert^\mathcal{F}_*\geq \lVert y\rVert_\mathcal{F}$. 	To prove the second inequality, take any $x \in c_{00}$ such that $\lVert x \rVert^\mathcal{F} = 1$. There exists partition $\mathcal{P} = \lbrace F_{1}, F_{2}, ..., F_{j} \rbrace$ of the support of $x$ for which the infimum in the definition of the quasi-norm is obtained, namely 
	\[
	\lVert x \rVert^{\mathcal{F}} = \sum_{i=1}^{j} \sup \limits_{k \in F_{i}} \lvert x(k) \rvert.
	\]

	Let $x'$ be defined by $x'(j)=a_i \cdot sgn(y(j))$ if $j\in F_i$,
	where $a_{i} = \sup \limits_{k \in F_{i}} \lvert x(k) \rvert$ (if $j\notin \bigcup_i F_i$, then let $x'(j)=0$). Then $\lVert x' \rVert^\mathcal{F} = \lVert x \rVert^\mathcal{F} = 1$ and

	\[
	\big \lvert \sum_{n \in \omega}x(n)y(n) \big \rvert \leq \sum_{n \in \omega} \lvert x(n)y(n) \rvert \leq  \sum_{n \in \omega} \lvert x'(n)y(n) \rvert.
	\]
Moreover
\[
	\sum_{n \in \omega} \lvert x'(n)y(n) \rvert 
= \sum_{i=1}^{j} \sum_{n \in F_{i}} \lvert x'(n)y(n) \rvert = \sum_{i=1}^{j} a_{i} \sum_{n \in F_{i}} \lvert y(n) \rvert \leq \sum_{i=1}^{j} a_{i} \lVert y \rVert_{\mathcal{F}} = \lVert y \rVert_{\mathcal{F}}
	\]
	which, as $c_{00}$ is dense in $X^\mathcal{F}$, implies that $\lVert y \rVert_{*}^{\mathcal{F}} \leq \lVert y \rVert_{\mathcal{F}}$ and finishes the proof. 
	\end{proof}

Unfortunately, one cannot deduce from Theorem \ref{final_isomorphism} that $X^\mathcal{F}$ and $X^*_\mathcal{F}$ are isomorphic. In fact, for some families $\mathcal{F}$ they are not.

\begin{eg}\label{tree} In this example we will consider finite dyadic trees, i.e. the sets  $T_N = \{0,1\}^{\leq N}$ of $0$-$1$ sequences of length at most $N$. Notice that, identifying elements of $T_N$ with
	natural numbers, using some fixed enumeration of $T_N$,
	we may think of $T_N$ as a subset of integers.  For $s\in \{0,1\}^N$ let $F_s = \{s_{\upharpoonright k} \colon k\leq N\}$ and let $\mathcal{F}_N$ be the hereditary closure of the family
	$\{F_s\colon s\in \{0,1\}^N\}$. So, $\mathcal{F}_N$ is the family of chains in $T_N$ and each $F_s$ is a maximal chain (a branch). For each
	$s\in \{0,1\}^N$ let $x_s$ be the vector in $\mathbb{R}^{T_N}$ given by $x_s = \chi_{F_s}$. Let $x = \sum_{s\in \{0,1\}^N} x_s$. Notice that $x(t) = |\{s\colon t \subseteq s\}|$. It can be checked by a simple induction (on $N$) that
	\[ \| x \|^{\mathcal{F}_N} = 2^N + 1\cdot 2^{N-1} + 2\cdot 2^{N-2} + 4\cdot 2^{N-3} + \dots + 2^{N-1}\cdot 1 \] 
	and so
	\[ \| x \|^{\mathcal{F}_N} = 2^N + N2^{N-1} = 2^N(1+N/2). \]
 Let $C>0$. Take $N$ so that $(1+N/2)>C$. Then
	\[ \| \sum_{s\in \{0,1\}^N} x_s \|^{\mathcal{F}_N} > C\cdot 2^N = C \sum_{s\in \{0,1\}^N} \| x_s \|^{\mathcal{F}_N}. \]

	So, at this point, for every $C>0$ we are able to find an example which violates the inequality from Remark \ref{normability} for the chosen $C$. Now, we will amalgamate all those $\mathcal{F}_N$'s into one example.

	For each $N$ fix an injection $k_N\colon T_N \to \mathbb{N}$ in such a way that the images $(k_N[T_N])_N$ is a partition of $\omega$. 
	Let \[ \mathcal{F} = \bigcup_N \{ k_N[F]\colon F\in \mathcal{F}_N\}. \] Then $\mathcal{F}$ is a compact hereditary family of subsets of $\omega$ covering $\omega$. But there is no $C>0$ such
	that 
	\[  \lVert \sum_{i\in A} x_i \rVert^\mathcal{F} \leq C \sum_{i\in A} \lVert x_i \rVert^\mathcal{F}  \]
	for every $A\subseteq \mathbb{N}$ and so, according to Remark \ref{normability}, $X^\mathcal{F}$ is not isomorphic to $X^*_\mathcal{F}$.
\end{eg}

The family from the above example was created to show that $X^\mathcal{F}$ may not be isomorphic to $X^*_\mathcal{F}$ for a compact family. Now, we will show that for the classical example of the compact family, the family of Schreier sets $\mathcal{S}$,
the same phenomenon occurs. The argument is more complicated. It indicates that if $\mathcal{F}$ is complicated enough, $X^\mathcal{F}$ is not isomorphic to $X^*_\mathcal{F}$. 

For any finite $A\subseteq\omega$ let $\phi(A)$ be the minimal number of consecutive Schreier sets in $A$ covering $A$. 

\begin{lem}\label{example} For any $N\in\omega$ there are sets $F_1,\dots,F_{2^N}\in\mathcal{S}$ so that for $x=\sum\limits_{j=1}^{2^N}\chi_{F_j}$ we have the following
\begin{enumerate}
	\item $x(i)\in\{2^r\colon r=0,\dots,N\}$ for any $i\in \mathrm{supp} (x)$,
    \item $\phi(A_r)\geq 2^{N-r}$, where $A_r=\{i\in\omega: x(i)=2^r\}$, for any $r=0,\dots,N$. 
\end{enumerate}
\end{lem}
\begin{proof}
	Fix $N\in\omega$. We shall again use the dyadic tree $T_N = \{0,1\}^{\leq N}$. This time we will assign to each element of $T_N$. First, we will linearize the inclusion ordering on $T_N$: define $\preceq$ on $T_N$ by
	\[ s \preceq t \ \ \  \mbox{ if } \ \ \ \big( t\subseteq s \mbox{ or } (t \mbox{ is incomparable with }s \mbox{ and } (s\cap t)^\frown 1 \subseteq s) \big). \]
Notice that $s \cap t$ is the longest element of $T_N$ which is extended both by $s$ and $t$.
Below we enclose a drawing of $T_3$ with the nodes enumerated according to $\preceq$.

	\begin{center}
\begin{tikzpicture}[auto, semithick, level/.style={sibling distance=60mm/#1}]
	\node {15}
        child { node {14}
            child { node {13}
                child { node {12} }
                child { node {11} }
            }
            child { node {10}
                child { node {9} }
                child { node {8} }
            }
        }
        child { node {7}
			child { node {6}
                child { node {5} }
                child { node {4} }
            }
            child { node {3}
                child { node {2} }
                child { node {1} }
            }
        };
\end{tikzpicture}
	\end{center}
\bigskip

	By $t_0$ we will denote the smallest element of $T_N$, i.e. the sequence constantly equal $1$. For $s\in T_N$, $s\ne t_0$ denote by $s^-$ be the immediate predecessor of $s$ and let $s'$ be the smallest, with respect to $\preceq$, descendant of
	$s$. Note that $s'$ is always a terminal node. For $r\leq N$ let $L_r$ be the $r$'s level of $T_N$, i.e. the set of elements $T_N$ of length $r$.

	For every $s \in T_N$ we will define an interval $I_s$ inductively, with respect to $\preceq$. Let $I_{t_0}=\{N+1\}$. If $s \ne t_0$ is a terminal node, then let $I_s$ be an interval of length $2\max |I_{s^-}|$ and such that $\min I_{s} \geq
	(2N+1)\max I_{s^-}$. For a non-terminal node $s$ let $I_s$ be an interval of length $|I_{s'}|$ and such that $\min I_{s} > \max I_{s^-}$. In this way we will get a sequence of intervals $(I_s)$ such that $s \preceq t$ iff $I_s < I_t$.

Each (maximal) branch $\mathcal{B}$ of $T_N$ induces a set $F_\mathcal{B}=\bigcup_{s\in\mathcal{B}}I_s$; the sets obtained in this way form the family $(F_j)_{j=1}^{2^N}$ defining the vector $x$ promised in the statement. First, notice that if $s\in
	L_{N-r}$, then $s$ belongs to $2^r$ many branches. So, $x$ satisfies the condition (1) of the statement. By the same reason we have $A_r=\bigcup_{s\in L_{N-r}} I_s$ for each $r$.

		We will check that the family $\{F_\mathcal{B}\colon \mathcal{B} \mbox{ is a branch}\}$ is as desired by proving two claims.
\medskip

\textbf{Claim 1.} $F_\mathcal{B} \in \mathcal{S}$ for every branch $\mathcal{B}$. 
\medskip
	
	Consider first the branch $\mathcal{B}_0$ containing $t_0$. By definition $|I_{s}|=1$ for each $s\in\mathcal{B}_0$, thus $|F_{\mathcal{B}_0}|\leq N+1$ (and, clearly, $\min F_{\mathcal{B}_0}
	= \min F_{t_0} = N+1$, so $F_{\mathcal{B}_0}\in \mathcal{S}$). Pick now branch $\mathcal{B}$ containing
	any other terminal node $s\in T_N$. By definition, $\min F_\mathcal{B}=\min I_s\geq (2N+1)\max I_{s^-}$. On the other hand, for every $t\in\mathcal{B}$, $t' \preceq s$ and so $|I_{t}|=|I_{t'}|\leq |I_s|= 2\max I_{s^-}$. It follows that  \[
		|F_\mathcal{B}|\leq (2N+1) \max I_{s^-} \leq \min F_\mathcal{B}, \] and so $F_\mathcal{B} \in \mathcal{S}$. 
\medskip

	\textbf{Claim 2.} $\phi(A_r)\geq |L_{N-r}|=2^{N-r}$ for each $r$. 
\medskip
	  
	Fix $r<N$ and an interval $I\in\mathcal{S}$ in $A_r = \bigcup_{s\in L_{N-r}} I_s$ (i.e. $I = J \cap A_r$, where $J$ is an interval). If for some $s_1 \preceq s_2\in L_{N-r}$ we have $I\cap I_{s_1}\neq\emptyset$, then $\max I<\max I_{s_2}$.  Indeed, notice that  $|I|\leq \max I_{s_1}$.  As $s_1\preceq
	s_2$, and $s_1$, $s_2$ belong to the same level, $s_1\preceq s_2'$, hence $2\max I_{s_1}\leq|I_{s_2'}|=|I_{s_2}|$ and so $\max I < \max I_{s_2}$. But this means that $\phi(A_r)\geq |L_{N-r}|=2^{N-r}$
\end{proof}

\begin{thm}
	The space $X^\mathcal{S}$ is not isomorphic to $\widehat{X}^\mathcal{S}$ (and, thus, it is not isomorphic to $X^*_\mathcal{S}$).
\end{thm}

\begin{proof}
For every $N\in\omega$ let $x_N=\sum\limits_{j=1}^{2^N}\chi_{F_j^N}$ be as in Lemma \ref{example}. 
Then,  we have $\|\chi_{F_j^N}\|^\mathcal{S}=1$ for any $j=1,\dots,2^N$ and $N\in\omega$, as each $F^N_j$ is a Schreier set. Consequently, 
\[ \sum_j \lVert \chi_{F_j^N} \rVert = 2^N. \]

	On the other hand, we will show that  $2^{-N}\|x_N\|^\mathcal{S} \to\infty$ as $N\to\infty$. Therefore, for each $C>0$ there is $N$ such that \[ \lVert \sum_j  \chi_{F_j^N} \rVert > C \cdot 2^N \]
	and so, by Remark \ref{normability}, we will be done.
	
Suppose, towards contradiction, that there is $M\in\omega$ with $\|x_N\|^\mathcal{S}\leq 2^{M+N}$. For a fixed $N\in\omega$ pick Schreier sets $(B_l)_l$ witnessing $\lVert x \rVert^\mathcal{S}$, i.e. such that
	$\|x_N\|^\mathcal{S}=\sum_{l} \max\{x_N(i)\colon i\in B_l\}$. Let \[l_r=|\{l: \max\{x_N(i)\colon i\in B_l\} = 2^r\}|\] for  $r=0,\dots,N$. Then 
\begin{equation}
 2^{N+M}\geq \sum_{r=0}^N l_r2^r.   
\end{equation}
%thus \begin{equation}    l_r\leq 2^{M+N-r}\text{ for any }r=0,\dots,N  \end{equation}
	On the other hand, as $A^r\subseteq B^r:=\bigcup\{B_l: \max \{x_N(i)\colon i\in B_l\} \geq 2^r \}$, 
\begin{equation}
    2^{N-r}\leq\phi(A^r)\leq \phi(B^r)\leq l_r+l_{r+1}+\dots+l_N \text{ for every }r=0,\dots,N-1.
\end{equation}
	The first inequality follows from Lemma \ref{example}, whereas the last one from the definition of $\phi$: the partition of a set into consecutive Schreier sets is the smallest in the sense of cardinality of all the partitions into Schreier sets.
	
In order to simplify the notation we write 
$m_r=l_{N-r}$, $r=0,\dots,N$. Then
\begin{equation}\label{eq1}
    \sum_{r=0}^N m_r2^{-r}\leq 2^M,
\end{equation}
thus
	\[    m_r\leq 2^{M+r} \text{ for every }r=0,\dots,N.\]
So, for $r> M+2$ we have
\begin{equation}\label{eq2}
	 m_0 + \dots + m_{r-M-2} \leq 2^M + \dots + 2^{M+r-M-2} = 2^M + \dots + 2^{r-2} \leq 2^{r-1}. 
\end{equation}
On the other hand $2^{N-r}\leq m_{N-r}+\dots+m_0$ for each $r=0,\dots,N$, that is
\begin{equation}\label{eq3}
    2^r\leq m_r+\dots+m_0 \text{ for any }r=0,\dots,N
\end{equation}
Therefore, for $r>M+2$ we have, by \eqref{eq2} and \eqref{eq3},
\begin{equation}\label{eq4}
    2^r\leq m_r+\dots+m_{r-M-2}+\dots+m_0\leq m_r+\dots +m_{r-M-1}+2^{r-1}.
\end{equation}
Hence, for every $r>M+2$ there is $i\in \{r-M-1,\dots,r\}$ with 
	\[ m_i\geq 2^{r-1}(M+2)^{-1}\] and so \[ m_i 2^{-i}\geq \frac{1}{2(M+2)}. \]
	By diving (sufficiently big) $N$ into intervals of size $M+2$ and subsequently using the above fact we see that 
\begin{equation}
	\sum_{r=0}^N m_r2^{-r}\geq \frac{N}{M+2} \cdot \frac{1}{2(M+2)} = \frac{N}{2(M+2)^2}
\end{equation}
which yields a contradiction with \eqref{eq1} for sufficiently big $N$.
\end{proof}
\begin{rem} In this article we are interested in the spaces induced by families of finite sets, but perhaps one can prove a more general version of the above theorem, working for families containing also infinite sets, but with $FIN(\lVert \cdot \rVert^\mathcal{F})$ instead of
	$X^\mathcal{F}$ (for $\mathcal{F}$ as in Theorem \ref{isomorphism}, those spaces coincide, see Theorem \ref{exh=fin}). E.g. notice that if $\mathcal{F}$ is the family of all subsets of $\omega$,
	then $X_\mathcal{F}$ is isometrically isomorphic to $\ell_1$ and $FIN(\lVert \cdot \rVert^\mathcal{F})$ is isometrically isomorphic to $\ell_\infty$.
\end{rem}

\begin{rem} There is also another approach of producing a norm on the dual space of combina\-torial-like spaces, than that considered above. In  \cite{Diana} Diana Ojeda-Aristizabal proposed a formula for the norm of the original space constructed by Tsirelson, which can be also adapted to the case of mixed Tsirelson spaces. The case of the original Tsirelson space is somewhat similar to the case of duals to combinatorial spaces; its (pre)dual norm can be derived from a precise formula, whereas the norm of the very space does not possess analogous expression. The formula proposed in \cite{Diana} is based on a dualization of the Figiel-Johnson norm (similar to our case), but yields a norm, instead of just a quasi-norm, via including in the definition of $\lVert x\rVert$ the expression $\inf\{\lVert y\rVert+\lVert z\rVert: y+z=x\}$, which forces the triangle inequality. As it is noted in \cite{Diana}, this definition does not permit to calculate the norm of a vector in a way similar to the case of Figiel-Johnson norm on the dual of the original Tsirelson space. In contrast, not including the above expression in our definition allows to work with the quasinorm on duals of combinatorial spaces, as it shown in the next section, with the price of the lack of the triangle inequality.
\end{rem}

\section{On $\ell_1$-saturation of $X^\mathcal{F}$'s} \label{l1-schur}

In this section we consider spaces $X^{\mathcal{F}}$ for a family $\mathcal{F}$ satisfying an additional condition (see definition \ref{large_family}). It is known that the dual space to $X_{\mathcal{S}}$ is an example of $\ell_1$-saturated space which does not enjoy the Schur property.  
	
We will show that the same holds for $X^{\mathcal{F}}$, for any compact family $\mathcal{F}$. %Clearly, by Theorem \ref{final_isomorphism} we are just reproving the above, known, fact. However, 
We think that our proof is considerably easier and it indicates
that studying $X^\mathcal{F}$ is easier than $X^*_\mathcal{F}$. This is why $X^\mathcal{F}$ may be helpful.

\begin{df} \label{schur_def}
	We say that a space $X$ has a \textit{Schur property} if every weakly null sequence is convergent to zero in the norm. 
	\end{df}

	\begin{df} \label{l1-def}
	A Banach space $X$ is \textit{$\ell_{1}$-saturated} if its every closed, infinitely dimensional subspace $E$ contains an isomorphic copy of $\ell_{1}$. Equivalently, if $E = \overline{span(x_{n})}$ for some basic sequence $(x_{n})$ in $X$, then there is a sequence $(y_{n})$ in $E$, which is equivalent to the standard basis of $\ell_{1}$.
	\end{df}

In the theory of Banach spaces the classical theorem of Rosenthal \cite{Rosenthal-l1} implies that every Banach space with Schur property is $\ell_{1}$-saturated. The question if this implication can be reversed was open for many years. Now, several
counter-examples are known, including the dual to the Schreier space. 
The lack of Schur property is quite straightforward. On the other hand, satisfying of $\ell_{1}$-saturation property follows form \cite{Galego}. 

We will show that Definition \ref{schur_def} and Definition \ref{l1-def} are also valid for quasi-Banach spaces and that $X^{\mathcal{S}}$ is an $\ell_{1}$-saturated quasi-Banach space which does not have Schur property. 

We will prove the lack of Schur property using Theorem \ref{isomorphism} for families satisfying certain property introduced by J.Lopez-Abad and S.Todorcevic in \cite{Abad-Todorcevic}.

\begin{df} \label{large_family}
A family $\mathcal{F}$ of subsets of $\omega$ is called \emph{large} when $\mathcal{F} \cap [M]^{n} \neq \emptyset$ for every infinite subset $M$ of $\omega$ and for every $n \in \omega$.  
\end{df}

\begin{prop} \label{Schur}
Suppose that $\mathcal{F}$ is a large family of finite subsets of $\omega$. Then $X^\mathcal{F}$ does not have the Schur property. 
\end{prop}

	\begin{proof}
		Consider the sequence $(e_{n})$, the standard Schauder basis. We claim that this sequence is weakly null, but it is not convergent to zero in the quasi-norm.

	Indeed, fix $\varphi \in (X^{\mathcal{F}})^{*}$ and denote $y(n) = \varphi(e_{n})$. By Theorem \ref{isomorphism} we have $y = (y(n)) \in FIN(\lVert \cdot \rVert_{\mathcal{F}})$, so $\lVert y \rVert_{\mathcal{F}} < \infty$. If $ \lim
	\limits_{n \rightarrow \infty} y(n) \neq 0$, then there is an infinite $M\subseteq \omega$ and $c>0$ such that $\lvert y(m) \rvert \geq c$ for each $m\in M$. By the assumption, for each $k\in \omega$ there is $F\in \mathcal{F}$ such that $F\subseteq M$,
	$|F| = k$ and so 
	$\mathlarger{\sum}_{i \in F} \lvert y(i) \rvert = c \cdot k$. Hence, $\lVert y \rVert_\mathcal{F} = \sup \limits_{F \in \mathcal{F}} \mathlarger{\sum}_{i \in F} \lvert y(i) \rvert = \infty$, a contradiction.

	On the other hand, for every $n \in \omega$ we have $\lVert e_{n} \rVert^{\mathcal{F}} = 1$, so $(e_{n})$ is not convergent to zero in the quasi-norm.
	\end{proof}
One can easily see that the Schreier family $\mathcal{S}$ is large in the sense of the Definition \ref{large_family}. Hence, in particular, $X^\mathcal{S}$ does not have the Schur property. 

%	As we can see the lack of Schur property is rather a frequent phenomenon for spaces of this type. 

%	The proof that $X^{\mathcal{S}}$ is $\ell_{1}$-saturated is, however, slightly more technical and uses more of the particular shape of the Schreier family $\mathcal{S}$. First, we introduce the definition which is necessary to prove that fact. 
\begin{df}\label{k-stable}
We say that the vector $x \in X^{\mathcal{F}}$ is \textit{$k$-stable}, $k\in\omega$, if $\lVert x \rVert_\infty\leq \dfrac{1}{2k} \lVert x \rVert ^{\mathcal{F}}$. %for every $F \in \mathcal{F}$ $\| P_{F}(x) \|^{\mathcal{F}} \leq \dfrac{1}{4k} \|x \|^{\mathcal{F}}$.
	\end{df}
Note that for any $k$-stable $x\in X^\mathcal{F}$ and any $F \in \mathcal{F}$ we have $\| P_{F}(x) \|^{\mathcal{F}} \leq \dfrac{1}{2k} \|x \|^{\mathcal{F}}$.

\begin{lem} \label{simple_lemma}
	Let $x_{1},x_{2} \in c_{00}$ be such that $\mathrm{supp}(x_1) < \mathrm{supp}(x_2)$. If $\mathcal{P}$ is a partition such that for every $P \in \mathcal{P}$ we have  $P \cap \mathrm{supp}(x_1) = \emptyset$ or $P \cap \mathrm{supp}(x_2) = \emptyset$, then \[ \lVert x_{1} + x_{2}
		 \rVert^{\mathcal{P}} = \lVert x_{1} \rVert^{\mathcal{P}} + \lVert x_{2} \rVert^{\mathcal{P}}.\]
	\end{lem}

	\begin{proof}
	Let $P \in \mathcal{P}$. Then, either $x_{1}$ or $x_{2}$ vanishes on $P$, so for each $k \in P$, $\lvert x_{1}(k) + x_{2}(k) \rvert = \lvert x_{1}(k) \rvert + \lvert x_{2}(k) \rvert$. It implies immediately that $\lVert x_{1} + x_{2} \rVert^{\mathcal{P}} = \lVert x_{1} \rVert^{\mathcal{P}} + \lVert x_{2} \rVert^{\mathcal{P}}$.
	\end{proof}

	\begin{prop} \label{stable-proposition}
		Let $x,y \in c_{00}$ be such that 
	\begin{enumerate}[(i)]
		\item $\mathrm{supp}(x) < \mathrm{supp}(y)$,
		\item $y$ is $k_{0}$-stable, where $k_{0} = \max \mathrm{supp}(x)$.
	\end{enumerate}
	Then for each scalar $\lambda$  \[\lVert x+ \lambda y \rVert^{\mathcal{F}} \geq \lVert x \rVert^{\mathcal{F}} + \frac{\lambda}{2}\lVert y \rVert^{\mathcal{F}}.\]
	\end{prop}

	\begin{proof} First, notice that if $y$ is $k$-stable, then $\lambda y$ is $k$-stable and so we may assume that $\lambda=1$.

	Since $x+y$ is finitely supported, there exists partition $\mathcal{P}$ such that %for which the infimum in the definition of the norm (\ref{Upper_norm}) is obtained for $x,y$ and $x+y$ respectively. Namely,
$\lVert x+y \rVert^{\mathcal{F}} = \lVert x+y \rVert^{\mathcal{P}}.$
		
		Let $A = \lbrace n \in \mathrm{supp}(y): \forall P \in \mathcal{P} \ (n \in P \Rightarrow P \cap \mathrm{supp}(x) \neq \emptyset) \rbrace$. Note that the sequences $x$, $P_{\omega \setminus A}(y)$ and the partition $\mathcal{P}$ satisfy the assumption of Lemma \ref{simple_lemma}. In addition, there are pairwise disjoint sets $F_1,...,F_{k_0} \in \mathcal{F}$ such that $A \subseteq \bigcup \limits_{i \leq k_0} F_i$. So, using the assumption of $k_0$-stability (see the remark after Definition \ref{k-stable}) and Lemma \ref{quasi-quasi} (a) we obtain
	\[
	\|P_{A}(y) \|^{\mathcal{F}} \leq \sum \limits_{i \leq k_0} \|P_{F_i}(y)\|^{\mathcal{F}} \leq \dfrac{1}{2} \|y \|^{\mathcal{F}}
	\]
		
		and thus
	\[
	\lVert x+y \rVert^{\mathcal{F}} = \lVert x+y \rVert^{\mathcal{P}} \geq \lVert x + P_{\omega \setminus A}(y) \rVert^{\mathcal{P}} = \lVert x \rVert^{\mathcal{P}} + \lVert P_{\omega \setminus A}(y) \rVert^{\mathcal{P}} 
 %\geq \lVert x \rVert^{\mathcal{P}} + \lVert y \rVert^{\mathcal{P}} 
 \geq \lVert x\rVert^\mathcal{F}+\lVert P_{\omega\setminus A}(y)\rVert^\mathcal{F}
 \geq \lVert x \rVert^{\mathcal{F}} + \dfrac{\lVert y \rVert^{\mathcal{F}}}{2}.  
\]
\end{proof}

\begin{thm}\label{ell1}
$X^{\mathcal{F}}$ is $\ell_{1}$-saturated.
\end{thm}

\begin{proof}
	Let $(x_{n})$ be a sequence in $X^{\mathcal{F}}$ and let $E$ be its closed subspace. We are going to show that $E$ contains an isomorphic copy of $\ell_1$. By the standard arguments (see e.g. \cite{BePe}) we may assume that $E =
	\overline{span(x_{n})}$, where for each $n\in \omega$ we have $x_{n} \in c_{00}$, $\lVert x_{n} \rVert^{\mathcal{F}} = 1$ and $\mathrm{supp}(x_n)$ is finite. Additionally, we assume that $\mathrm{supp}(x_n)<\mathrm{supp}(x_{n+1})$ for each $n\in\omega$.
 
 %(poprawiłam, bo nie potrzebujemy pozniej $F_j$) there exists a finite subset $F_{n}$ such that $\mathrm{supp}(x_{n}) \subseteq F_{n}$. Additionally, we assume that $F_{n} < F_{n+1}$. 

	It is enough to construct a sequence $(y_n)$ of unit vectors in $E$ which will be equivalent to the standard $\ell_1$-basis, i.e. such that for each sequence $(\lambda_i)_{i\leq n}$ of scalars \[ \lVert \sum_{i=1}^{n}\lambda_{i}y_{i} \rVert^{\mathcal{F}} \geq \dfrac{1}{2} \mathlarger{\sum}_{i=1}^{n} \lvert \lambda_{i} \rvert. \]
The sequence $(y_n)$ will be of the form of a \emph{block sequence} of $(x_n)$, i.e.  $y_{n} = \sum \limits_{i=l_n+1}^{l_{n+1}} \alpha_{i}x_{n_{i}}$, $n\in\omega$, for some %$k \in \omega$ and 
increasing sequences $(n_{i})_i$ and $(l_n)_n$ of natural numbers. 

	We define by induction sequences of natural numbers $(k_{n})$, $(l_{n})$ and the sequence of vectors $(y_{n})$. Let $l_{1} = 1$, $y_{1} = x_{1}$ and $k_{1} = \max \mathrm{supp}(x_{1})$. Next, let 
\begin{itemize}
	\item[(1)] $l_{n+1}\in \omega$ be such that
$\mathlarger{\sum}_{i=l_{n}+1}^{l_{n+1}}x_{i}$ is $k_{n}$-stable,
\item[(2)]  $y_{n+1} = \dfrac{\mathlarger{\sum}_{i=l_{n}+1}^{l_{n+1}}x_{i}}{\lVert \mathlarger{\sum}_{i=l_{n}+1}^{l_{n+1}}x_{i}\rVert^\mathcal{F}}$ and
\item[(3)]  $k_{n+1} = \max \mathrm{supp}(y_{n+1})$.
\end{itemize}

We will show that we are able to perform such construction. Only condition (1) needs some explanations.

\textbf{Claim.} For each $l\in \omega$ and each $k\in \omega$ there is $L>l$ such that $\mathlarger{\sum}_{i=l}^{L}x_{i}$ is $k$-stable.

%	Indeed, denote  $x = \sum \limits_{i=l}^\infty x_i$ and notice that $\lVert x \rVert^\mathcal{F} = \infty$. Otherwise, $x \in FIN(\lVert \cdot \rVert^\mathcal{F})$ and so $x\in X^\mathcal{F}$, by Proposition \ref{exh=fin}. But this is impossible since $\lVert x_i \rVert^\mathcal{F} =
%1$ for each $i$. 
%%So, we may find $L$ big enough so that
%%\[ \lVert \sum_{i=l}^L x_i \rVert^\mathcal{S} > 2k. \]
%%But if $|A|\leq k$, then (since $x_i$'s have disjoint supports)  
%%\[ \lVert P_A(\sum_{i=l}^L x_i) \rVert^\mathcal{S} \leq |A|\lVert x_i\rVert^\mathcal{S} \leq k \]
%%and so $L$ is as desired.

Suppose it is not true. Then there are $k,l \in \omega$ such that for each $L > l$ %there is $F_L \in \mathcal{F}$ for which $$ \|P_{F_L}(\sum \limits_{i=l}^{L}x_i)\|^{\mathcal{F}} 
\[
\lVert \sum_{i=l}^L x_i\rVert_\infty > \dfrac{1}{2k} \|\sum \limits_{i=l}^{L}x_i\|^{\mathcal{F}}.\]
Denote $x = \sum \limits_{i=l}^{\infty} x_i$. Then, for every $L$ we have $\|P_L(x)\|^{\mathcal{F}} < 2k \|x\|_\infty \leq 2k$, and so by lower semicontinuity $\|x\|^{\mathcal{F}} \leq 2k$. Thus, $x \in FIN(\|\cdot\|^{\mathcal{F}})$. On the other
	hand, for every $i \in \omega$ we have $\|x_{i}\|^{\mathcal{F}} = 1$, so $\|P_{\omega \setminus L}(x)\|^{\mathcal{F}} \geq 1$ for every $L$, which means that $x \notin X^{\mathcal{F}}$. This is a contradiction with Proposition \ref{exh=fin}.

%and notice that $\|x \|^{\mathcal{F}} = \infty$. Otherwise, $x \in FIN(\| \cdot\ \|^{\mathcal{F}})$ and so $x\in X^\mathcal{F}$, by Proposition \ref{exh=fin}. 

%But this is impossible by Theorem \ref{5.4}, since $\lVert x_i \rVert^\mathcal{F} = 1$. On the other hand, it is easy to notice that $\lVert \sum\limits_{i=l}^L x_i\rVert_\infty\leq 1$ 
%for every $F \in \mathcal{F}$, $\|P_{F}(\sum \limits_{i=l}^{L}x_i)\|^{\mathcal{F}} \leq 1$ 
%and so we get a contradiction by the lower semicontinuity of $\lVert\cdot\rVert^\mathcal{F}$. 

\bigskip

Having this construction, fix $n \in \omega$ and a sequence $(\lambda_{i})_{i \leq n}$. Of course, $\lVert y_{i} \rVert^{\mathcal{F}} = 1$ for each $i$ and subsequently using Proposition \ref{stable-proposition} we have 
\[
\lVert \sum_{i=1}^{n}\lambda_{i}y_{i} \rVert^{\mathcal{F}} = \lVert \sum_{i=1}^{n-1}\lambda_{i}y_{i} + \lambda_{n}y_{n} \rVert^{\mathcal{F}} \geq \lVert \sum_{i=1}^{n-1}\lambda_{i}y_{i} \rVert^{\mathcal{F}} + \frac{\lvert \lambda_{n}\rvert}{2} \geq
... \geq \lvert \lambda_{1} \rvert + \frac{1}{2}\mathlarger{\sum}_{i=2}^{n} \lvert \lambda_{i} \rvert \geq \frac{1}{2} \mathlarger{\sum}_{i=1}^{n} \lvert \lambda_{i} \rvert
\]
and so $(y_n)$ is equivalent to the standard $\ell_1$-basis.
\end{proof}

\bibliographystyle{alpha}
\bibliography{quasi-banach-2}

\end{document}